\documentclass[12pt]{amsart}
\usepackage[all,cmtip]{xy} 
\usepackage{subcaption}
\usepackage{multirow}
\usepackage{amssymb}
\usepackage{xcolor}
\usepackage{algorithmicx}
\usepackage{paralist} 
\usepackage{mathtools}
\usepackage[english]{babel}
\usepackage{graphics}
\usepackage{graphicx}
\graphicspath{ {images/} }
\usepackage{epstopdf}

\usepackage{amsmath}
\usepackage{mathrsfs}
\usepackage{tikz-cd}
\usetikzlibrary{shapes,shadows,arrows}
\usepackage{cite}
\usepackage{placeins}
\usepackage{hyperref}
\usepackage{tcolorbox}
\calclayout                   

\theoremstyle{plain} 

\newtheorem{dummy}{anything}[section]
\newtheorem{theorem}[dummy]{Theorem}

\newtheorem{lemma}[dummy]{Lemma}

\newtheorem{proposition}[dummy]{Proposition}

\theoremstyle{definition}

\newtheorem{remark}[dummy]{Remark}

\def\:{\mkern 1.2mu \colon}
\newcommand{\mmatrix}[4]{\left (\vcenter
	{\xymatrix@C-2pc@R-2pc{#1&#2\\#3&#4} } \right )}

\DeclareMathOperator{\Mod}{Mod}
   
\numberwithin{equation}{section}

\makeindex 

\begin{document}
	\title[Minimal Generation of Mapping Class Groups]
	{Minimal Generation of Mapping Class Groups: A Survey of the Nonorientable Case }
    
    \subjclass[2020]{Primary: 57K20; Secondary: 20F38, 20F05 }
	\keywords{nonorientable surface, mapping class group, minimal generating set, torsion element, involution, commutator.}
	
	\author{ T{\"{u}}l{\.I}n Altun{\"{o}}z, Mehmetc{\.I}k Pamuk, Oguz Yildiz}
		
	\address{Faculty of Engineering, Ba\c{s}kent University, Ankara, Turkey} 
\email{tulinaltunoz@baskent.edu.tr} 
\address{Department of Mathematics, Middle East Technical University,
 Ankara, Turkey}
 \email{mpamuk@metu.edu.tr}
 \address{Department of Mathematics, Middle East Technical University,
 Ankara, Turkey}
  \email{oguzyildiz16@gmail.com}
		
	\date{\today}	
	\begin{abstract}
This chapter provides a comprehensive survey of foundational results and recent advances concerning minimal generating sets for the mapping class group of a nonorientable surface, $\Mod(N_{g})$, and its index-two twist subgroup, $\mathcal{T}_{g}$. Although the theory for orientable surfaces is well established, the nonorientable case presents unique challenges due to the presence of crosscaps, thus requiring generators beyond Dehn twists. We show that, for a sufficiently large genus $g$, both $\Mod(N_{g})$ and $\mathcal{T}_{g}$ are generated by two elements, which is the minimum possible number. The survey details various types of generating sets, including those composed of torsions, involutions, and commutators, illustrating the geometric and algebraic interplay. We unify foundational work with modern breakthroughs and extend results to punctured surfaces, $\Mod(N_{g,p})$, providing explicit generators, relations, and proof sketches with an emphasis on geometric intuition.
\end{abstract}
	
	
	\maketitle	

\section{Introduction}

The mapping class group\index{mapping class group} (MCG) of a surface, its group of self-homeomorphisms up to isotopy, stands as a fundamental object in low-dimensional topology. For decades, the theory of mapping class groups has focused primarily on orientable surfaces. Here, the theory is a model of mathematical elegance, founded on the pioneering work of Dehn~\cite{dehn}, Nielsen~\cite{nielsen}, and Lickorish~\cite{lickorish}. A key breakthrough was the proof of finite generation, with Dehn twists\index{Dehn twist} playing the central role. This foundational work led to deep connections with Teichmüller theory, moduli spaces of algebraic curves, and symplectic and arithmetic groups. The elegance of this theory extends to its minimal generating sets\index{generating set!minimal}. A particularly striking result by Maclachlan~\cite{maclachlan} showed that the generation of the mapping class group by torsion elements\index{torsion element} implies that the moduli space of algebraic curves is simply-connected. Building on this, Wajnryb~\cite{wajnryb} later proved that for an orientable surface of genus greater than one, the mapping class group can be generated by just two elements. These results on minimal generation are not merely numerical curiosities. The quest to find the smallest set of generators often leads to the discovery of interesting new elements and intricate relations within the group, offering deeper insights into its algebraic and geometric structure. This well-understood world of twists on orientable surfaces provides the essential backdrop for the wilder, less-explored frontier of their nonorientable counterparts.

When one turns to nonorientable surfaces\index{surface!nonorientable}, however, the landscape becomes wilder and less explored. In the 1970s, D.R.J. Chillingworth pioneered the study of MCGs for nonorientable surfaces, identifying finite generating sets for surfaces with boundary~\cite{chillingworth}. His work revealed stark contrasts with orientable cases: for instance, Dehn twists alone do not generate the full group, necessitating new types of generators like crosscap slides\index{crosscap slide} (or Y-homeomorphisms\index{Y-homeomorphism}). This departure from the orientable case makes the MCGs of nonorientable surfaces, denoted $\Mod(N_g)$, a compelling field of study for several reasons. They provide a crucial testing ground for understanding how fundamental changes in surface topology radically alter algebraic group structure. Furthermore, these groups act on the moduli spaces of real algebraic curves, forging a deep link to real algebraic geometry. Beyond pure mathematics, nonorientable surfaces appear in physical contexts like string theory, and their mapping class groups are essential for understanding their symmetries and quantization.

Subsequent advances in the 1990s--2000s by mathematicians such as Wajnryb, Korkmaz, and Szepietowski expanded the field. Szepietowski derived explicit finite presentations for MCGs of closed nonorientable surfaces~\cite{Szepietowski1994}, while more recent work by Paris, Margalit, and others has explored connections to curve complexes and automorphism groups of free groups~\cite{Paris2013}. A particularly rich theme in the theory of $\Mod(N_g)$ is that of minimal generation\index{minimal generation}. Fundamental questions include: What is the smallest number of elements needed to generate the group? Can it be generated by torsion elements? By involutions\index{involution}? Are there generating sets that consist entirely of commutators\index{commutator}? The answers are subtle and often depend delicately on the genus $g$.

Recent advances have led to significant breakthroughs.
\begin{itemize}
    \item For $g \geq 19$, $\Mod(N_g)$ can be generated by just two elements~\cite{altunoz-pamuk-yildiz}.
    \item For $g \geq 26$, it can be generated by three involutions~\cite{altunoz-pamuk-yildiz}.
    \item For all $g \neq 4$, it admits generating sets consisting of three torsion elements~\cite{du, lesniak-szepietowski}.
    \item For large $g$, these torsion elements can be taken to be conjugate~\cite{lesniak-szepietowski}.
\end{itemize}

These results have analogues for the \emph{twist subgroup}\index{twist subgroup} $\mathcal{T}_g \subset \Mod(N_g)$, the index-two subgroup generated solely by Dehn twists about two-sided curves. This subgroup retains many features of the orientable MCG while reflecting the nonorientable topology in its algebraic structure. The twist subgroup is perfect\index{perfect group} for $g \geq 7$, and recent work has produced minimal generating sets of torsions, involutions, and commutators.

This survey charts the landscape of generating sets for $\Mod(N_g)$ and its index-two twist subgroup, $\mathcal{T}_g$. We begin by recounting the foundational results that established finite generation and delve into the central theme of minimal generation. We will explore the progression of results that ultimately show that for a sufficiently large genus $g$, both $\Mod(N_g)$ and $\mathcal{T}_g$ can be generated by just two elements, the minimum possible. Along the way, we will examine the nature of these generators, focusing on generating sets composed entirely of torsions, involutions, or commutators, which reveal the intricate interplay between the algebra of the group and the geometry of the surface. Finally, we will survey the current state of knowledge for the more complex case of nonorientable surfaces with punctures\index{puncture}, $\Mod(N_{g,p})$.

\section{Foundations of Nonorientable Surfaces and Mapping Class Groups}

Let \(N_{g}\) be a closed connected nonorientable surface of genus\index{genus!nonorientable} \(g\), where the genus is defined as the number of crosscaps\index{crosscap}. The \textit{mapping class group}\index{mapping class group!definition} \(\operatorname{Mod}(N_{g})\) is the group of isotopy classes of all self-homeomorphisms of \(N_{g}\). For small genera:
\begin{itemize}
    \item \(\operatorname{Mod}(N_{1})\) is trivial.
    \item \(\operatorname{Mod}(N_{2}) \cong \mathbb{Z}/2\mathbb{Z}\)~\cite{lickorishnon1}.
    \item For \(g \geq 3\), the structure becomes richly complex.
\end{itemize}

\subsection*{Key Geometric Elements}
Represent \(N_g\) as a sphere with \(g\) crosscaps (disks with antipodal boundary identifications) as shown in Figure~\ref{NG}. On such surfaces:
\begin{itemize}
    \item A simple closed curve is \textit{one-sided}\index{curve!one-sided} if its regular neighborhood is a Möbius band (e.g., curves through crosscaps).
    \item A curve is \textit{two-sided}\index{curve!two-sided} if its neighborhood is an annulus (admits Dehn twists).
\end{itemize}
\begin{figure}[h]
\begin{center}
\scalebox{0.35}{\includegraphics{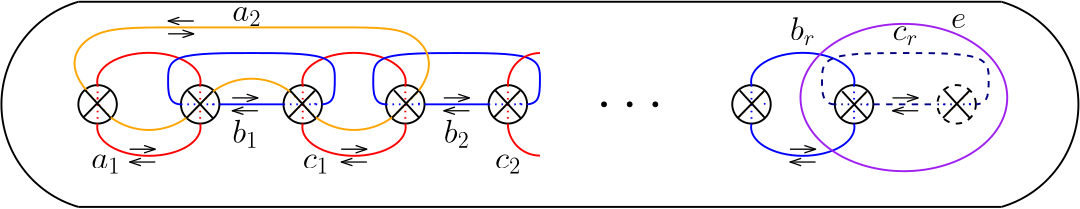}}
\caption{A standard model of the surface $N_g$ for $g=2r$ or $g=2r+1$. The figure shows the curves $a_{1}$, $a_{2}$, $b_{i}$, $c_{i}$, and the curve $e$ which bounds a Klein bottle with a hole formed by the last two crosscaps. Note that the curve $c_{r}$ does not exist when $g$ is odd.}
\label{NG}
\end{center}
\end{figure}

\subsection*{Generators and Relations}
Lickorish~\cite{lickorishnon1,lickorishnon2} first proved that \(\operatorname{Mod}(N_g)\) is finitely generated. The generating set consists of:
\begin{itemize}
    \item \textit{Dehn twists}\index{Dehn twist} \(t_\alpha\) about two-sided curves (Figure~\ref{twist}).
    \item The \textit{crosscap slide}\index{crosscap slide} \(Y_{\mu,\alpha}\) (Y-homeomorphism\index{Y-homeomorphism}). 
\end{itemize}
Crosscap slides arise when interchanging consecutive crosscaps in a Klein bottle subsurface \(K\) (see Figure~\ref{YHOM}). Algebraically, \(Y_{\mu,\alpha} = t_\alpha u_{\mu,\alpha}\) where \(u_{\mu,\alpha}\) is a crosscap transposition\index{crosscap transposition} satisfying \(u_{\mu,\alpha}^2 = t_{\partial K}\).
\begin{figure}[h]
\begin{center}
\scalebox{0.35}{\includegraphics{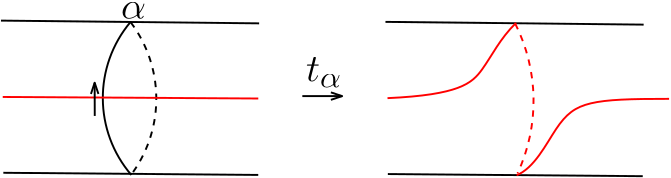}}
\caption{A description of a Dehn twist.}
\label{twist}
\end{center}
\end{figure}

\begin{figure}[h]
\begin{center}
\scalebox{0.3}{\includegraphics{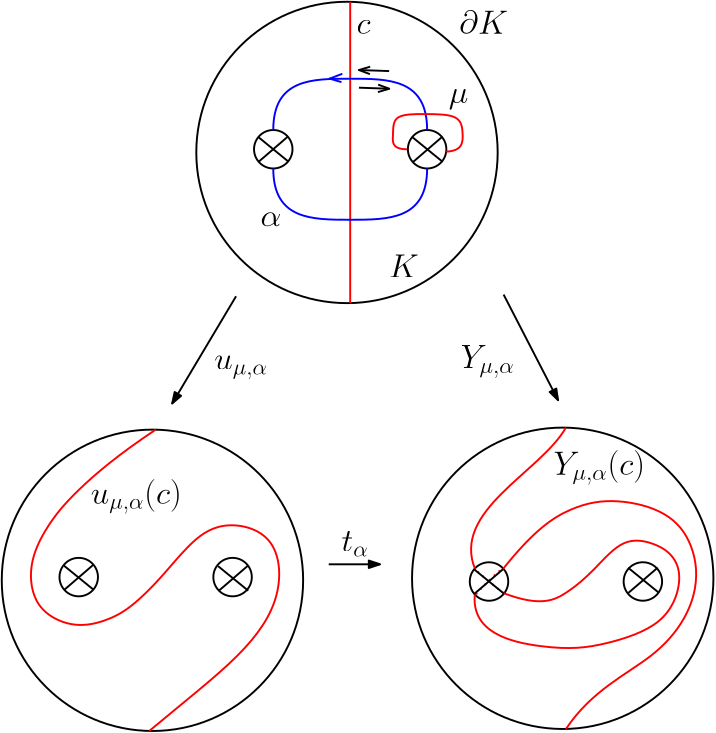}}
\caption{Crosscap transposition \(u_{\mu,\alpha}\) and crosscap slide \(Y_{\mu,\alpha}\).}
\label{YHOM}
\end{center}
\end{figure}
\par
\noindent

Conjugation relations govern generators:
\[
f t_\alpha f^{-1} = t_{f(\alpha)}^s, \quad 
f Y_{\mu,\alpha} f^{-1} = Y_{f(\mu),f(\alpha)}
\]
where \(s = \pm 1\) depends on whether \(f\) preserves or reverses local orientation. Throughout this chapter, we write the conjugation $fgf^{-1}$ as $g^{f}$ for any mapping classes $f, g$.
\noindent
Fundamental relations include:
\[
Y_{\mu_1,\alpha}^{-1} Y_{\mu_2,\alpha} = Y_{\mu_2,\alpha} Y_{\mu_1,\alpha}^{-1} = t_\alpha^2,
\]
where the one-sided curves $\mu_1$ and $\mu_2$ are as in Figure~\ref{yhom_relation}.
\begin{figure}[h]
\begin{center}
\scalebox{0.45}{\includegraphics{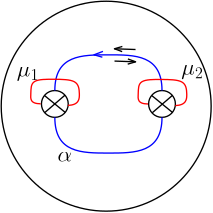}}
\caption{Relations on crosscap slides.}
\label{yhom_relation}
\end{center}
\end{figure}
\subsection*{Twist Subgroup and Homology Action}
The \textit{twist subgroup}\index{twist subgroup} \(\mathcal{T}_g \subset \operatorname{Mod}(N_g)\), generated by all Dehn twists, has index $2$~\cite{lickorishnon2,stukow}. The action of $\mathrm{Mod}(N_g)$ on the first homology group\index{homology group} with $\mathbb{Z}_2$ coefficients yields a surjective homomorphism $\mathrm{Mod}(N_g) \to \mathrm{Aut}(H_1(N_g; \mathbb{Z}_2))$ that preserves the mod-2 intersection form~\cite{mpin}. The image of this representation is isomorphic to~\cite{sz3,korkmaz1998}:
\[
\begin{cases} 
\operatorname{Sp}(2h; \mathbb{Z}_2) & g = 2h+1 \\
\operatorname{Sp}(2h; \mathbb{Z}_2) \ltimes \mathbb{Z}_2^{2h+1} & g = 2h+2 
\end{cases}
\]
inducing an epimorphism onto the symplectic group\index{symplectic group} \(\operatorname{Sp}\left(2\lfloor \frac{g-1}{2} \rfloor; \mathbb{Z}_2\right)\). Thus, generators of \(\operatorname{Mod}(N_g)\) project to generators of this symplectic quotient. This connection to classical groups is a key algebraic tool for analyzing the structure of $\mathrm{Mod}(N_g)$.

To provide a more complete picture, it is useful to introduce the broader context of nonorientable surfaces with punctures and boundary components, as many results for closed surfaces extend to these more general settings. Let $N=N_{g,p}^{m}$ denote a connected nonorientable surface of genus $g$ with $m$ boundary components and $p$ punctures. To construct a representation of $N$, we begin with an orientable connected surface $\Sigma=\Sigma_{g_1,p}^{g_2+m}$ of genus $g_1$ having $g_2 + m$ boundary components and $p$ punctures, where $2g_1 + g_2 = g$ and $g_2 \geq 1$. The nonorientable surface $N$ is then visualized by adding crosscaps to $g_2$ of the boundary components in a standard model of $\Sigma$ (see Figure~\ref{modelexample_non} for an example).
\begin{figure}[h]
\begin{center}
\scalebox{0.32}{\includegraphics{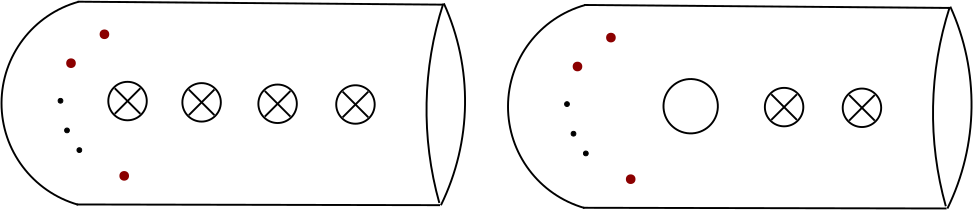}}
\caption{Two different representations of $N_{4,p}^{1}$.}
\label{modelexample_non}
\end{center}
\end{figure}

\section{Generating sets for the mapping class group of nonorientable surfaces}
Let $N_g$ be a nonorientable genus $g$ surface shown in Figure~\ref{NG}. In this section, we present some known finite generating sets\index{generating set} for the mapping class group $\Mod(N_g)$. We also consider the minimal number of elements required for such generating sets. Using Lickorish’s partial results~\cite{lickorishnon1,lickorishnon2}, Chillingworth~\cite{chillingworth} obtained a finite generating set whose elements generate all homeomorphisms of $N_g$. Later, using some ideas from the theory of covering spaces, Birman and Chillingworth~\cite{birman-chillingworth} showed how generators for $\Mod(N_g)$ could be obtained from those of $\Mod(\Sigma_g)$. Szepietowski~\cite{szepietowski} proved that $\Mod(N_g)$ can be generated by involutions using the generators given in~\cite{birman-chillingworth}. He later proved that $\Mod(N_g)$ can be generated by three elements and also four involutions~\cite{szepietowski1}. In~\cite{altunoz-pamuk-yildiz}, this result was improved by proving that for $g \ge 19$, $\Mod(N_g)$ can be generated by two elements, and for $g \ge 26$, by three involutions. Recently, if $g\neq 4$, a generating set consisting of three torsion elements for the mapping class group $\Mod(N_g)$ was obtained in~\cite{lesniak-szepietowski}.
\subsection{Generating sets involving Dehn twists and a crosscap slide}
It is known that $\Mod(N_g)$ is generated by Dehn twists together with one crosscap slide~\cite{lickorishnon1,lickorishnon2}. Chillingworth~\cite{chillingworth} obtained a finite collection of generators for $\Mod(N_g)$. For $g> 4$, Szepietowski~\cite{szepietowski1} decreased the number of Dehn twist generators and proved the following theorem:

\begin{theorem}\label{thmnon1}
The mapping class group $\Mod(N_{g})$ is generated by a crosscap slide and the Dehn twists about the curves in the set $\{t_{a_1}, t_{a_2}\} \cup \{t_{b_i}\}_{i=1}^r \cup \{t_{c_i}\}_{i=1}^{k}$, where $k=r$ if $g$ is even and $k=r-1$ if $g$ is odd.
\end{theorem}
If $g = 4$, then the generating set given by Chillingworth coincides with the generating set given in Theorem~\ref{thmnon1}. In the case $g=3$, $\Mod(N_3)$ is generated by the Dehn twists $\{ t_{a_1}, t_{b_1} \}$ and a crosscap slide~\cite{birman-chillingworth}.
\subsection{Minimal number of generators for $\Mod(N_g)$ }
For orientable surfaces of genus $ \geq 2 $, it is known that the mapping class group can be generated by two elements~\cite{wajnryb}. The nonorientable setting introduces additional complexity compared to the orientable case. A key distinction is the behavior of crosscap slides: unlike in the orientable case, these cannot be represented as compositions of Dehn twists. The abelianization\index{abelianization} of $\Mod(N_g)$ shows it cannot be cyclic, implying that any generating set must contain at least two elements. In this section, we present some known minimal generating sets for the mapping class group $\Mod(N_g)$. We give a brief proof of one of these results to explain the method used.

For \( g = 3 \), set \( a: = t_{a_1} \) and \( b := t_{b_1} \). The mapping class group \(\Mod(N_3)\) has a presentation with generators \( a, b \), a crosscap slide \( y \), and relations:
\[
aba = bab, \quad (aba)^4 = 1, \quad yay = a^{-1}, \quad yby = b^{-1}, \quad y^2 = 1
\]
as shown in \cite{birman-chillingworth}. Korkmaz \cite{korkmaz2012} later proved that \( \Mod(N_3) \) is generated by \( a \) and \( by \), which follows from the identities:
\[
(by)a^{-1}(by)^{-1} = bab^{-1}, \quad a(bab^{-1})a^{-1} = b.
\]
Consider the product of $g-1$ Dehn twists:
\[
S = \begin{cases}
t_{c_r}t_{b_r}t_{c_{r-1}}t_{b_{r-1}}\cdots t_{c_1}t_{b_1}t_{a_1} & \text{if } g \text{ is even} \\
t_{b_r}t_{c_{r-1}}t_{b_{r-1}}t_{c_{r-2}}\cdots t_{c_1}t_{b_1}t_{a_1} & \text{if } g \text{ is odd}
\end{cases}
\]
One can verify that $S$ acts on the curves as:
\[
S^{-1}(a_1)=b_1,\quad S^{-1}(b_i)=c_i,\quad S^{-1}(c_i)=b_{i+1} \text{ for } i\geq1.
\]
For any $\alpha\in \{a_1,b_i,c_i\}$, a subgroup $G\subset\Mod(N_g)$ containing both $t_{\alpha}$ and $S$ must contain all Dehn twists about the curves $ \{a_1,b_i,c_i\}$, making $S$ crucial for the proof of the next theorem. Let $y=t_e^2$ be a crosscap slide, and let $\phi$ be a homeomorphism supported in the Klein bottle bounded by $e$ such that $y$ can be expressed as $\phi t_{c_r}$ (or $\phi t_{b_r}$ when $g=2r+2$). Szepietowski~\cite{szepietowski1} proved that $\Mod(N_g)$ can be generated by three elements. He proved the following theorem.
\begin{theorem}\label{thmnon2}
Let $N_g$ be a closed nonorientable surface of genus $g=2r+1$ or $g=2r+2$, where $r\geq2$. The mapping class group $\Mod(N_g)$ is generated by the collection
\begin{itemize}
    \item $\{ t_{a_2}, S, \phi \}$ ,
    \item $\{ t_{a_2}, S, \phi t_{b_2} \}$ if $g$ is even and $g\geq 8$,
     \item $\{ t_{a_2}, S, \phi t_{b_1} \}$ if $g=6$, 
       \item $\{ t_{a_1}, St_{a_2}, \phi \}$ if $g=4$.
\end{itemize}
\end{theorem}

The authors~\cite{altunoz-pamuk-yildiz} proved that for $g\geq 19$, two elements suffice to generate the mapping class group $\Mod(N_g)$, which is the smallest possible set of generators. This result is proved briefly below.
\begin{figure}[h]
\begin{center}
\scalebox{0.33}{\includegraphics{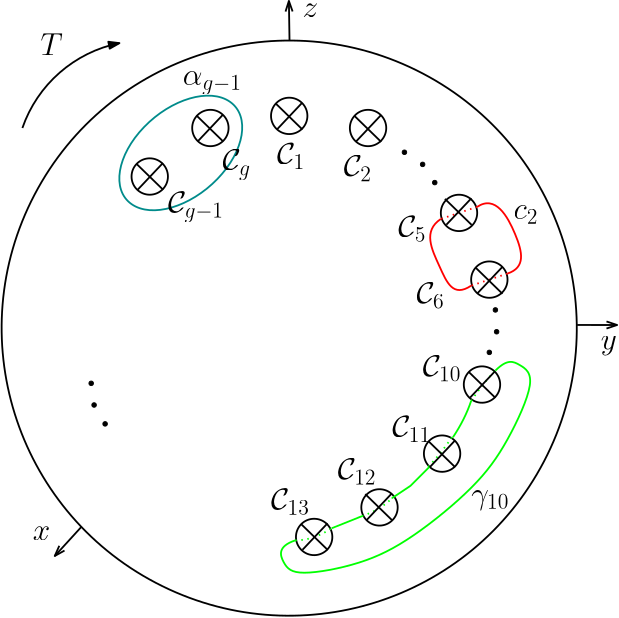}}
\caption{The rotation $T$ and the curves $\alpha_{g-1}, c_2$ and $\gamma_{10}$ on $N_g$.}
\label{rotationT}
\end{center}
\end{figure}
Let us consider the model for $N_g$ shown in Figure~\ref{rotationT}, in which the crosscaps are labeled as $\mathcal{C}_i$. The rotation $T$ by $\frac{2\pi}{g}$ about the $x$-axis maps the crosscap $\mathcal{C}_i$ to $\mathcal{C}_{i+1}$ for $i=1,\ldots,g-1$ and $\mathcal{C}_g$ to $\mathcal{C}_1$, as depicted in Figure~\ref{rotationT}. We denote by $u_i$ the crosscap transposition defined on the one-holed Klein bottle whose boundary curve is $\alpha_i$ shown in Figure~\ref{nongenerators1}. Since  $T$ maps $\alpha_i$ to $\alpha_{i+1}$, the conjugation by $T$ yields $Tu_iT^{-1} = u_{i+1}$.
\begin{figure}[h]
\begin{center}
\scalebox{0.3}{\includegraphics{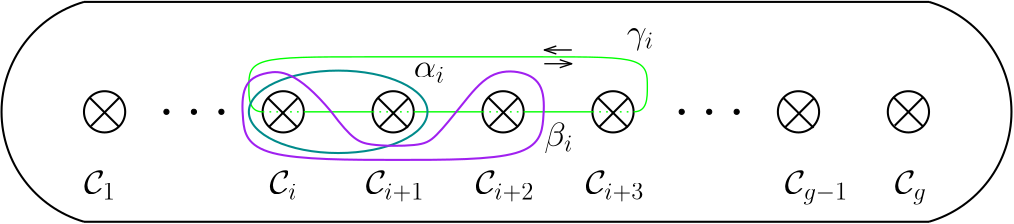}}
\caption{The curves $\alpha_i, \beta_i$ and $\gamma_i$ on $N_g$.}
\label{nongenerators1}
\end{center}
\end{figure}
\begin{theorem}\label{theoremnon3}
 The mapping class group $\Mod(N_g)$ for $g\geq7$ can be generated by the elements $T$, $t_{a_1}t_{a_2}^{-1}$, $t_{b_1}t_{ b_2}^{-1}$ and $u_{g-1}$.
\end{theorem}

The proof proceeds by verifying that the standard generators from Theorem~\ref{thmnon1} are contained in the subgroup generated by $T$, $t_{a_1}t_{a_2}^{-1}$, $t_{b_1}t_{ b_2}^{-1}$ and $u_{g-1}$. The lantern relation\index{lantern relation} also plays a vital role in this proof.
\begin{theorem}\label{theoremnon4}
 The mapping class group $\Mod(N_g)$ for $g\geq19$ is generated by the elements $T$ and $u_{g-1}t_{\gamma_{10}}t_{c_2}^{-1}$.
\end{theorem}
\begin{proof}
    Denote by $G_1$ the element $u_{g-1}t_{\gamma_{10}}t_{c_2}^{-1}$. Let $G$ be the subgroup of $\Mod(N_g)$ generated by $T$ and $G_1$. By Theorem~\ref{theoremnon3}, proving $G = \Mod(N_g)$ reduces to showing that the subgroup $G$ contains $t_{a_1}t_{a_2}^{-1}$, $t_{b_1}t_{b_2}^{-1}$ and $u_{g-1}$. The diffeomorphism $T^{-4}$ sends the curves $(\alpha_{g-1},\gamma_{10}, c_2)$ to $(\alpha_{g-5},\gamma_{6},a_1)$. Thus, the subgroup $G$ contains the conjugated element
\begin{itemize}
    \item $G_2:=G_1^{T^{-4}}=u_{g-5}t_{\gamma_{6}}t_{a_1}^{-1}$.
\end{itemize}
It can be verified that $G_2G_1^{-1}(\alpha_{g-5},\gamma_{6}, a_1)=(\alpha_{g-5},c_{2}, a_1)$. By the conjugation relation, we get the following element:
\begin{itemize}
    \item $G_3:=G_2^{G_2G_1^{-1}}=u_{g-5}t_{c_{2}}t_{a_1}^{-1} \in G$.
\end{itemize}
Thus, $G$ contains 
\begin{itemize}
    \item $t_{\gamma_{6}}t_{c_2}^{-1}=G_2G_3^{-1}$,
\end{itemize}
and so, we have
\begin{itemize}
\item $t_{\gamma_{10}}t_{c_4}^{-1}=(t_{\gamma_{6}}t_{c_2}^{-1})^{T^4} \in G$.
\end{itemize}
Moreover, the following elements are shown to be in $G$:
\begin{itemize}
\item $G_4:=(t_{c_4}t_{\gamma_{10}}^{-1})G_1=u_{g-1}t_{c_{4}}t_{c_2}^{-1}$,
\item $G_5:=G_4^{T^{-1}}=u_{g-2}t_{b_{4}}t_{b_2}^{-1}$,
\item $G_6:=G_3^{G_4G_5}=u_{g-5}t_{b_{2}}t_{a_1}^{-1}$.
\end{itemize}
From this, $G$ contains
\begin{itemize}
\item $t_{b_{2}}t_{c_2}^{-1}=G_6G_3^{-1}$,
\item $t_{c_{2}}t_{b_2}^{-1}=(t_{b_{2}}t_{c_2}^{-1})^T$,
\item $t_{b_{2}}t_{b_3}^{-1}=(t_{b_{2}}t_{c_2}^{-1})(t_{c_{2}}t_{b_3}^{-1})$,
\end{itemize}
which implies that
\begin{itemize}
\item $t_{b_{1}}t_{b_2}^{-1}=(t_{b_{2}}t_{b_3}^{-1})^T\in G$.
\end{itemize}
One can also show that the following elements belong to $G$: \begin{itemize}
\item $t_{c_{3}}t_{b_4}^{-1}=(t_{c_{2}}t_{b_3}^{-1})^{T^2}$,
\item $t_{\gamma_{8}}t_{c_3}^{-1}=(t_{\gamma_{10}}t_{c_4}^{-1})^{T^{-2}}$,
\item $t_{\gamma_{8}}t_{b_4}^{-1}=(t_{\gamma_{8}}t_{c_3}^{-1})(t_{c_{3}}t_{b_4}^{-1})$.
\end{itemize}
From these, it follows that
\begin{itemize}
\item $t_{a_{2}}t_{a_1}^{-1}=(t_{\gamma_{1}}t_{a_1}^{-1})=(t_{\gamma_{8}}t_{b_4}^{-1})^{T^{-7}}\in G$.
\end{itemize}
 Following the proof of Theorem~\ref{theoremnon3} (see~\cite[Theorem 2.1]{altunoz-pamuk-yildiz}), the subgroup $G$ contains the elements $t_{a_1}$, $t_{a_2}$, $t_{b_i}$ and $t_{c_i}$ for $i=1,\ldots,r$. Then,
\begin{itemize}
\item 
$t_{\gamma_{10}}=t_{a_2}^{T^9}\in G.$
\end{itemize}
  It follows that the crosscap transposition
 \begin{itemize}
\item $u_{g-1}=G_1(t_{c_2}t_{\gamma_{10}}^{-1})\in G$,
\end{itemize}
concluding the proof.
\end{proof}
 
 \begin{remark}\label{nonremark}
For $g=4$, the abelianization $H_1(\Mod(N_4))$ is isomorphic to $(\mathbb{Z}/2\mathbb{Z})^3$~\cite{korkmaz1998}, which implies that $\Mod(N_4)$ cannot be generated by two elements.
 \end{remark}
\subsection{Torsion generators of $\Mod(N_g)$}\label{torsionnonorientable}

For $g \geq 7$ the normal closure of any finite-order element in $\Mod(N_g)$ is either the subgroup of $\Mod(N_g)$ generated by Dehn twists, known as the \textit{twist subgroup}\index{twist subgroup}, or the entire group~\cite{lesniak}. If $2\leq g \leq 6$, $\Mod(N_g)$ fails to be normally generated by one element since it has a non-cyclic abelianization~\cite{korkmaz1998}. For $g=4k+3 \geq 7$, Du~\cite{du} obtained a generating set\index{generating set!torsion} for $\Mod{(N_g)}$ consisting of two involutions and one element of order $2g$. Related work by Leśniak and Szepietowski~\cite{lesniak-szepietowski} further obtained torsion generators for $\Mod{(N_g)}$. We state their main results here and outline a proof for one of these results. Define four torsion elements of $\Mod(N_g)$:

\begin{itemize}
\item $s= 
    \begin{cases}
t_{a_1}t_{b_1}t_{c_1}t_{b_2}t_{c_2}\cdots t_{b_r}t_{c_r} & \text{if } g=2r+2, \\ t_{a_1}t_{b_1}t_{c_1}t_{b_2}t_{c_2}\cdots t_{c_{r-1}}t_{b_r} & \text{if } g=2r+1.
 \end{cases}$
    \item $s^{\prime}= 
    \begin{cases}
t_{a_1}^2t_{b_1}t_{c_1}t_{b_2}t_{c_2}\cdots t_{b_r}t_{c_r} & \text{if } g=2r+2, \\ t_{a_1}t_{b_1}t_{c_1}t_{b_2}t_{c_2}\cdots t_{c_{r-1}}t_{b_r} & \text{if } g=2r+1.
 \end{cases}$
\end{itemize}
The element $s$ has order $g$ when $g$ is even and order $2g$ when $g$ is odd, while $s^{\prime}$ has order $g-1$ for even $g$ and $2(g-1)$ for odd $g$~\cite{Paris-Szepietowski2015}. Following our earlier definition, $u_i$ is the crosscap transposition associated with a one-holed Klein bottle whose boundary curve is $\alpha_i$ (Figure~\ref{nongenerators1}).
\begin{itemize}
\item $t= u_1 u_2\cdots u_{g-1}$
\item $t^{\prime}= u_2 u_3\cdots u_{g-1}$    
\end{itemize}
The elements $t$ and $t^{\prime}$ have orders $g$ and $g-1$, respectively. Note that $t$ and $t^{\prime}$ are conjugates of the rotation shown in~\cite[Figure 4]{lesniak-szepietowski}. It is easy to verify that the elements $s$ and $t$ map the curves $ \lbrace a_1,b_i,c_i \rbrace $ to $\lbrace b_1,c_i,b_{i+1}\rbrace $, respectively. It is known that $\Mod(N_3)$ is generated by the Dehn twists $\{ t_{a_1}, t_{b_1} \}$ and a crosscap slide~\cite{birman-chillingworth}. It follows that $\Mod(N_3)$ is generated by  $s=t_{a_1}t_{b_1}$ of order six, $s^{\prime}=t_{a_1}^2t_{b_1}$ of order four, and the crosscap transposition $u_2$ of order two. For $g=4$, the abelianization of $\Mod(N_4)$ is $\mathbb{Z}_2^3$ (see Remark~\ref{nonremark}). However, it remains an open question whether $\Mod(N_4)$ can be generated by three torsions. Let 
$y= 
    \begin{cases}
t_{b_r}u_{g-1} & \text{if } g \text{ is odd}, \\ t_{c_r}u_{g-1} & \text{if } g \text{ is even}.
 \end{cases}$
\begin{theorem}\label{torsionnonthm}
    For $g\neq 4$, the mapping class group $\Mod(N_g)$ is generated by three torsion elements.
\end{theorem}
\begin{proof}
 Suppose first that $g=5$. We will show that $\Mod(N_5)$ is generated by the torsion elements $s$, $st_{a_2}$ (which is of order $6$) and $y^{-1} t^{\prime} y$. Let $G=\langle s, st_{a_2}, y^{-1} t^{\prime} y\rangle$. Observe that $t_{a_2}$ is clearly contained in $G$. One can verify that
\begin{itemize}
    \item $t_{b_2}=t_{a_2}^{(st_{a_2})^{-1}}\in G.$
\end{itemize}
Conjugating by $s$ shows that the subgroup $G$ contains the Dehn twists $t_{a_1}$, $t_{b_1}$ and $t_{c_1}$. For $g\geq 6$, we define the element
$$x= 
    \begin{cases}
y^{-1}t_{b_1}t_{c_1}t_{b_2}t_{a_2} & \text{if } g=6, \\ 
t_{b_r}u_{g-2}t_{b_1}t_{c_1}t_{b_2}t_{a_2}& \text{if } g=2r+1\geq 7,\\
t_{c_r}u_{g-2}t_{b_1}t_{c_1}t_{b_2}t_{a_2}& \text{if } g=2r+2\geq 8.
    \end{cases}$$
One can check that $x(b_2)=a_2$ and $x(b_1)=c_1$. Suppose that $g=2r+2\geq 6$. Let $G=\langle s, s^{\prime}, xtx^{-1}  \rangle$. It is easy to see that the Dehn twist $t_{a_1}=s^{\prime}s^{-1}$ is in $G$. By the action of $s$, the Dehn twists $t_{a_1}$, $t_{b_i}$ and $t_{c_i}$ are contained in $G$. To show that $t_{a_2}\in G$, consider the diffeomorphism $xt^2x^{-1}$ satisfying
\begin{itemize}
    \item $t_{a_2}=t_{c_1}^{(xt^2x)^{-1}}\in G$.
\end{itemize}
If $g=2r+1\geq 7$, it can be shown that $\Mod(N_g)$ is generated by the torsion elements $s$, $s^{\prime}$ and $xtx^{-1}$ following the same argument as above.
\end{proof} 
The next result says that for a sufficiently large $g$, the mapping class group  $\Mod(N_g)$ is generated by three conjugate torsion elements. Leśniak and Szepietowski follow Lanier's techniques~\cite{lanier2018} for the proof (see~\cite{lesniak-szepietowski} for complete details).
\begin{theorem}\label{lesniak-szepietowski-non}
Let $k\geq 12$ be even and $p$ be odd. For a nonorientable surface of genus $g=pk+2q(k-1)$ or $g=pk+2q(k-1)+1$ with $q\in \mathbb Z_{> 0}$, $\Mod(N_g)$ is generated by three conjugate elements of order $k$.
\end{theorem}

\subsection{Involution generators of $\Mod(N_g)$ }

This section focuses especially on generating sets\index{generating set!involution} consisting entirely of involutions, as these elements provide both algebraic simplicity and intuitive geometric interpretation, most naturally interpreted as reflection-like transformations. Comparing with the orientable case, the nonorientable case involves algebraic and geometric challenges due to the presence of crosscaps. It is well known that any group generated by two involutions is isomorphic to a quotient of a dihedral group\index{dihedral group}. However, the mapping class group $\Mod(N_g)$ contains non-abelian free subgroups, which cannot occur in dihedral group quotients. This structural obstruction proves that $\Mod(N_g)$ cannot be generated by two involutions, implying three as the strict lower bound for involution generators. Szepietowski~\cite{szepietowski} proved that $\Mod(N_g)$ can be generated by involutions using the generators given in~\cite{birman-chillingworth}. He later proved that $\Mod(N_g)$ can be generated by three elements and also four involutions~\cite{szepietowski1}. In~\cite{altunoz-pamuk-yildiz}, his result was improved by proving that for $g\geq 26$, $\Mod(N_g)$ can be generated by three involutions. For $g=3$, it is known that $\Mod(N_3)$ admits a known presentation with generators $a,b$ and a crosscap slide $y$ subject to the relations $aba=bab$, $(aba)^4=1$, $yay=a^{-1}$, $yby=b^{-1}$ and $y^2=1$~\cite{birman-chillingworth}. Using this presentation, it follows directly that $\Mod(N_3)$ can be generated by the three involutions $y$, $ay$, and $by$~\cite{korkmaz2012}. Below, we give a proof that $\Mod(N_g)$ can be generated by three involutions if $g\geq 26$. For any elements $x,y \in G$ with  $\rho x \rho = y$, where $\rho$ is an involution, the product $\rho x y^{-1}$ is an involution.
\begin{figure}[h]
\begin{center}
\scalebox{0.28}{\includegraphics{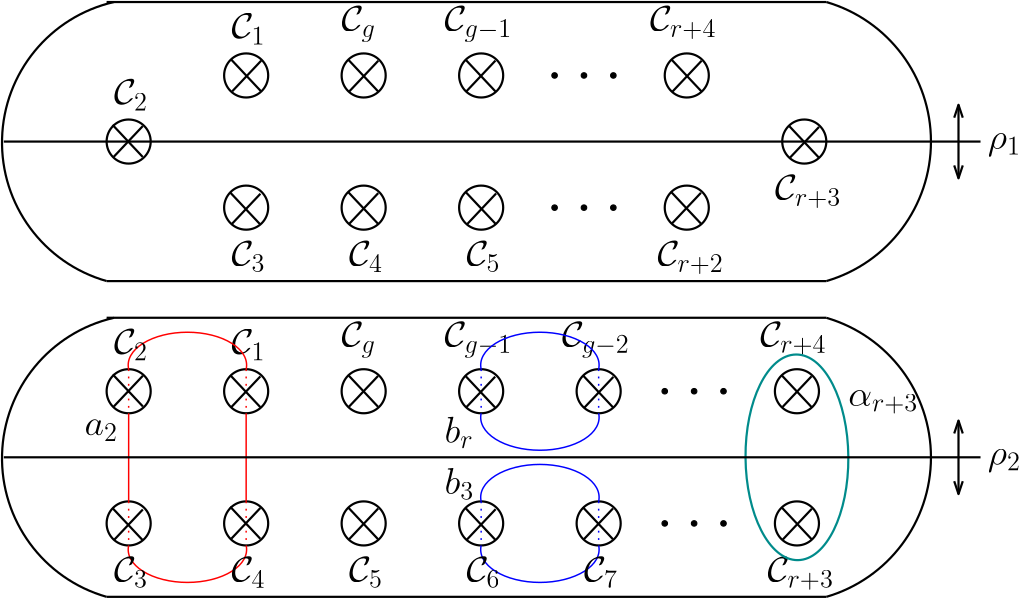}}
\caption{The reflections $\rho_1,\rho_2$ for $g=2r+2$.}
\label{nonrhoeven}
\end{center}
\end{figure}
\begin{figure}[h]
\begin{center}
\scalebox{0.13}{\includegraphics{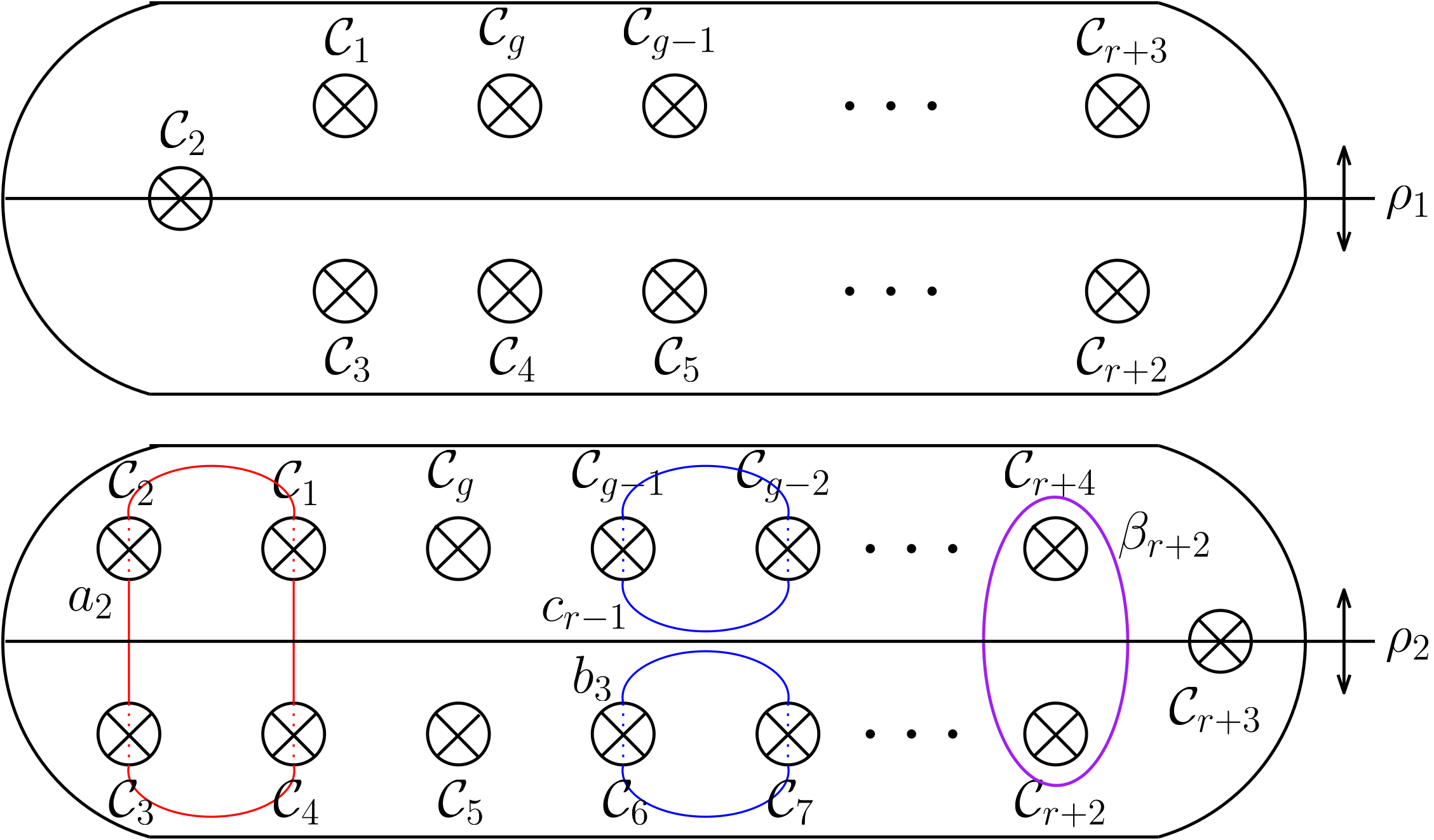}}
\caption{The reflections $\rho_1,\rho_2$ for $g=2r+1$.}
\label{nonrhoodd}
\end{center}
\end{figure}
For $g=2r+2$, we consider the two reflection generators $\rho_1$ and $\rho_2$ depicted in Figure~\ref{nonrhoeven}. These elements correspond to reflections across the planes shown in Figure~\ref{nonrhoeven}. The rotation $T$ shown in Figure~\ref{rotationT} is then expressed as the composition $\rho_2\rho_1$. Recall that $u_i$ is the crosscap transposition supported on the one-holed Klein bottle whose boundary curve is $\alpha_i$ shown in Figure~\ref{nongenerators1}. The rotation $T$ maps $\alpha_i$ to $\alpha_{i+1}$, which implies that $Tu_iT^{-1} = u_{i+1}$. Observe that the reflection $\rho_2$ preserves $a_2$ and $\alpha_{r+3}$, and maps $b_r$ to $b_3$ while reversing local orientations. This implies the conjugation relations
\[
\rho_2t_{a_2}\rho_2=t_{a_2}^{-1}, \quad \rho_2t_{b_2}\rho_2=t_{b_3}^{-1} \textrm{  and  } \rho_2 u_{r+3}\rho_2=u_{r+3}^{-1}.
\]
Combining these results, we get that $\rho_2t_{a_2}t_{b_r}t_{b_3}u_{r+3}$ is an involution. 

\begin{theorem}\label{theoremnon5} \cite{altunoz-pamuk-yildiz}
For a nonorientable surface $N_g$ of genus $g=2r+2\geq26$, the mapping class group $\Mod(N_g)$ is generated by the involutions $\rho_1$, $\rho_2$ and $\rho_2t_{a_2}t_{b_r}t_{b_3}u_{r+3}$.
\end{theorem}

\begin{proof}
The proof relies on showing that the subgroup $G$ generated by the three given involutions contains a known generating set, such as the one from Theorem~\ref{theoremnon3}. Let $G_1 = t_{a_2}t_{b_r}t_{b_3} u_{r+3}$, so the generators of $G$ are $\rho_1$, $\rho_2$, and the involution $\rho_2 G_1$. The rotation $T = \rho_2\rho_1$ is therefore in $G$, as is the element $G_1$ itself. The argument, similar in technique to the proof of Theorem~\ref{theoremnon4}, proceeds by constructing specific products and conjugates of the generators to isolate individual Dehn twists and crosscap transpositions. For instance, by finding appropriate products, one can show that $t_{c_2}t_{b_3}^{-1} \in G$. Conjugating this element by powers of the rotation $T$ is a key step that yields all elements of the form $t_{b_i}t_{c_i}^{-1}$ and $t_{c_i}t_{b_{i+1}}^{-1}$, which in turn generates the chain of twists $t_{b_i}t_{b_{i+1}}^{-1}$. With further algebraic manipulation, one can recover the remaining required generators, $t_{a_1}t_{a_2}^{-1}$ and the crosscap transposition $u_{g-1}$. Since all generators from Theorem~\ref{theoremnon3} can be constructed, we conclude that $G = \mathrm{Mod}(N_g)$. For a detailed construction of the intermediate elements, we refer the reader to \cite{altunoz-pamuk-yildiz}.
\end{proof}

For $g=2r+1$, we use the reflections $\rho_1$ and $\rho_2$ depicted in Figure~\ref{nonrhoodd} so that $T=\rho_2 \rho_1$. In the following theorem, the element $v_i$ represents the crosscap transposition supported on the one-holed Klein bottle bounded by the curve $\beta_i$ shown in Figure~\ref{nongenerators1}. Since the rotation $T$ satisfies $T(\beta_i)=\beta_{i+1}$, we get the conjugation relation $Tv_iT^{-1}=v_{i+1}$. Since
\[
\rho_2(a_2)=a_2 \textrm{ and }    \rho_2(c_{r-1})=b_3,
\]
and since the reflection $\rho_2$ reverses the local orientations of the two-sided simple closed curves $a_2$, $c_{r-1}$ and $\beta_{r+2}$, we have
\[
\rho_2t_{a_2}\rho_2=t_{a_2}^{-1}, \quad \rho_2t_{c_{r-1}}\rho_2=t_{b_3}^{-1} \textrm{ and } \rho_2v_{r+2}\rho_2=v_{r+2}^{-1}.
\]
A straightforward computation shows that the element $\rho_2t_{a_2}t_{c_{r-1}}t_{b_3}v_{r+2}$ is an involution.
\begin{theorem}\label{theoremnon6}
For a nonorientable surface $N_g$ of genus $g=2r+1\geq27$, the mapping class group $\Mod(N_g)$ is generated by the involutions $\rho_1$, $\rho_2$ and $\rho_2t_{a_2}t_{c_{r-1}}t_{b_3}v_{r+2}$.
\end{theorem}
The proof of Theorem~\ref{theoremnon6} follows similarly to Theorem~\ref{theoremnon5}, so we omit it here to avoid repetition (for details, see the proof of ~\cite[Theorem 4.2]{altunoz-pamuk-yildiz}).
\section{Generating sets for $\Mod(N_{g,p}$)}
Let $N_{g,p}$ denote a nonorientable surface\index{surface!punctured} of genus $g$ with $p$ punctures $P=\{z_1,z_2,\ldots,z_p\}$, which is obtained from a closed nonorientable surface $N_g$ by deleting $p$ points. When $p=0$, we simply write $N_{g}$. We present $N_{g,p}$ with $p$ punctures and $g$ crosscaps in Figure~\ref{nonpunctured_curves}. The \textit{mapping class group} $\Mod(N_{g,p})$ of the surface $N_{g,p}$ is defined to be the group of isotopy classes of all diffeomorphisms $N_{g,p} \to N_{g,p}$ preserving the set $P$. A fundamental question about this group is to find a finite generating set. In the orientable case, classical constructions provide explicit generating sets, typically composed of Dehn twists and half-twists\index{half-twist}. However, the nonorientable case presents additional complexities due to the lack of global orientation and the existence of new types of diffeomorphisms, such as crosscap slides and puncture slides (to be defined later). Szepietowski~\cite[Theorem~3]{szepietowski} showed that $\Mod(N_{g, p})$ is generated by involutions. Later, Yoshihara~\cite{yoshihara2022} gave finite generating sets for $\Mod(N_{g,p})$ with  $p \geq 0$. He proved that it can be generated by eight involutions if $g \geq 13$ is odd and eleven involutions if $g \geq 14$ is even. Furthermore, for $p \geq 1$ and $g \geq 14$, $\Mod(N_{g,p})$ admits smaller generating sets consisting of five elements or six involutions. These bounds were recently improved by the authors~\cite{altunoz-pamuk-yildiz5}. In this section, we present these generating sets.

Let $\Mod_0(N_{g,p})$ denote the subgroup of the mapping class group $\Mod(N_{g,p})$ consisting of elements that fix the set $P$ pointwise. This subgroup is a normal subgroup of $\Mod(N_{g,p})$ and has index $p!$. There is an exact sequence:
 \[
1\longrightarrow \Mod_{0}(N_{g,p})\longrightarrow \Mod(N_{g,p}) \longrightarrow Sym_{p}\longrightarrow 1,
\]
where $Sym_p$ is the symmetric group on the set $\lbrace1,2,\ldots,p\rbrace$ and the last projection in the sequence is given by the restriction of the isotopy class of a diffeomorphism to its action on the punctures. Equivalently, since $\Mod(N_{g,p})$ acts on the set of punctures, $\Mod_0(N_{g,p})$ is precisely the kernel of this action. It consists of all mapping classes that fix each puncture.

In the mapping class group $\Mod(N_{g,p})$, there exists a different type of diffeomorphism called a \textit{puncture slide}\index{puncture slide}. Consider a Möbius band $M$ with a puncture and a one-sided simple closed curve $\alpha$ (as depicted in Figure~\ref{punctureslide}). The puncture slide along $\alpha$, denoted by $v_{\alpha}$, is obtained by sliding the puncture once around the curve $\alpha$, while keeping the boundary points of $M$ pointwise fixed. The action of the puncture slide $v_{\alpha}$ on the interval $c$ is depicted in Figure~\ref{punctureslide}.
\begin{figure}[h]
\begin{center}
\scalebox{0.35}{\includegraphics{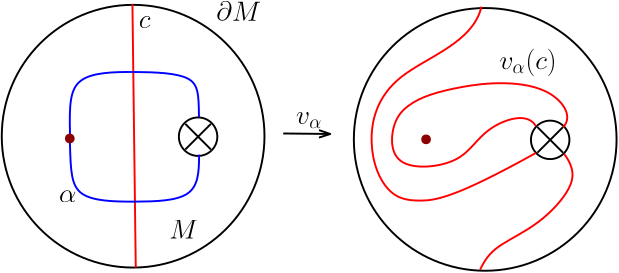}}
\caption{The puncture slide $v_{\alpha}$.}
\label{punctureslide}
\end{center}
\end{figure}
 We give Korkmaz's generating set for $\Mod_{0}(N_{g,p})$~\cite{korkmaz2002}. Consider the collection $\mathcal{C}$ of simple closed curves defined as follows:
\begin{itemize}
    \item For odd genus $g=2r+1$:
    \begin{flalign*}
    \mathcal{C} &= \{a_i, b_i, f_i \mid i=1,\ldots,r\} \cup \{c_j \mid j=1,\ldots,r-1\} \cup \{e_k \mid k=1,\ldots,p-1\}&
    \end{flalign*}

    \item For even genus $g=2r+2$:
    \begin{flalign*}
    \mathcal{C} &= \{a_i, b_i, f_i, c_i \mid i=1,\ldots,r\} \cup \{e_k \mid k=1,\ldots,p-1\},&
    \end{flalign*}
\end{itemize}

where the curves are shown in Figure~\ref{generators_non_punctured}.
\begin{remark}
    There exists a homeomorphism of $N_g$ that maps the curves shown in Figure~\ref{NG} to the correspondingly labeled curves in Figure~\ref{generators_non_punctured}. For a detailed description of this homeomorphism, see~\cite{chillingworth}.
\end{remark}

\begin{theorem}[{\cite{korkmaz2002}}]\label{puncture_non_thm}
For $g \geq 3$, the mapping class group $\Mod_{0}(N_{g,p})$ is generated by
\begin{itemize}
\item[(i)] Dehn twists $\{t_c\}_{c\in\mathcal{C}}$, the crosscap slide $y$, and puncture slides $\{v_{g,i}\}_{i=1}^p$ when $g$ is odd,
\item[(ii)]  all generators from case (i) together with the additional puncture slides $\{v_{g-1,i}\}_{i=1}^p$ when $g$ is even.
\end{itemize}
\end{theorem}

\begin{figure}[h]
\begin{center}
\scalebox{0.28}{\includegraphics{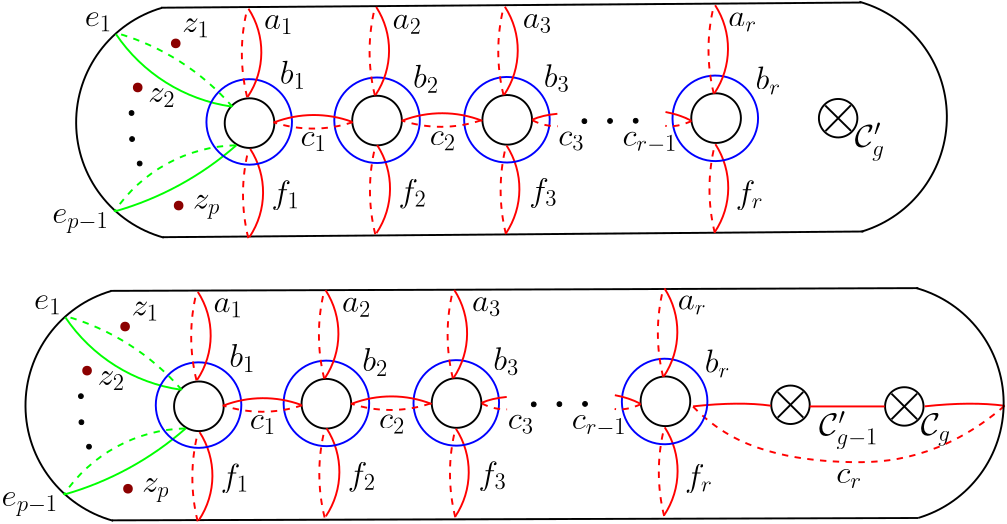}}
\caption{The Dehn twist generators for $\Mod_0(N_{g,p})$.}
\label{generators_non_punctured}
\end{center}
\end{figure}

Yoshihara uses different approaches for odd and even genus. For $g = 2r + 1 \geq 13$, he constructs involutions $\sigma, \tau, I, J, W$ and composite involutions: $\rho_1 = \tau t_{b_1}$, $\rho_2 = \tau v_{g,1}$ and $\rho_3=Wy$, where $y$ is a crosscap slide (see~\cite{yoshihara2022} for explicit definitions of these involutions and $y$). The proof has two key steps. The first is that the group generated by these eight involutions is contained in Korkmaz's generating set for $\Mod_{0}(N_{g,p})$, as given in Theorem~\ref{puncture_non_thm}. The second step is to prove surjectivity onto $Sym_{p}$, which follows from the action of the involutions $\sigma, \tau, W$. For $g = 2r + 2 \geq 14$, the proof requires additional elements: the involution $K$ and the composite involution elements $\rho_4 = Jt_{f_{r+1}}$ when $r=2k+1$ (when $r=2k$, interchange the curve $f_{r+1}$ so that it is preserved under $J$), $\rho_5 = Jt_{c_r}$ (for the curve $c_r$, see~\cite[Figure 19]{yoshihara2022}). Similar to the odd case, the subgroup of $\Mod(N_{g,p})$ generated by the eleven involutions $\sigma$, $\tau$, $I$, $J$, $W$, $K$, and $\rho_i$ ($i=1,2,3,4,5$) contains the generators for $\Mod_{0}(N_{g,p})$. The proof of surjectivity onto $Sym_{p}$ follows from the same elements used in the odd case.

Then we have the following theorem:
\begin{theorem}~\cite{yoshihara2022}
Let $p$ be a non-negative integer. Then, for $g$ odd with $g \geq 13$, the mapping class group $\Mod(N_{g,p})$ is generated by eight involutions. For $g$ even with $g \geq 14$, $\Mod(N_{g,p})$ is generated by eleven involutions.
\end{theorem}

In~\cite{altunoz-pamuk-yildiz5}, an improvement of Yoshihara's result was presented. We now state the new result, which says that $\Mod(N_{g,p})$ can be generated by five elements or six involutions for $g\geq 14$ and $p\geq 1$. Below, we give the generating sets and outline the key ideas of the proof.
\begin{figure}[h]
\begin{center}
\scalebox{0.38}{\includegraphics{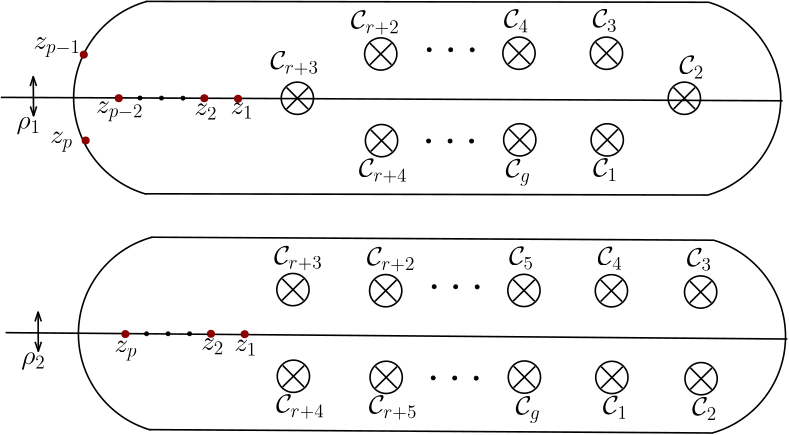}}
\caption{The reflections $\rho_1,\rho_2$ on $N_{g,p}$ for $g=2r+2$.}
\label{nonpunctured_evenrho1-2}
\end{center}
\end{figure}

\begin{figure}[h]
\begin{center}
\scalebox{0.32}{\includegraphics{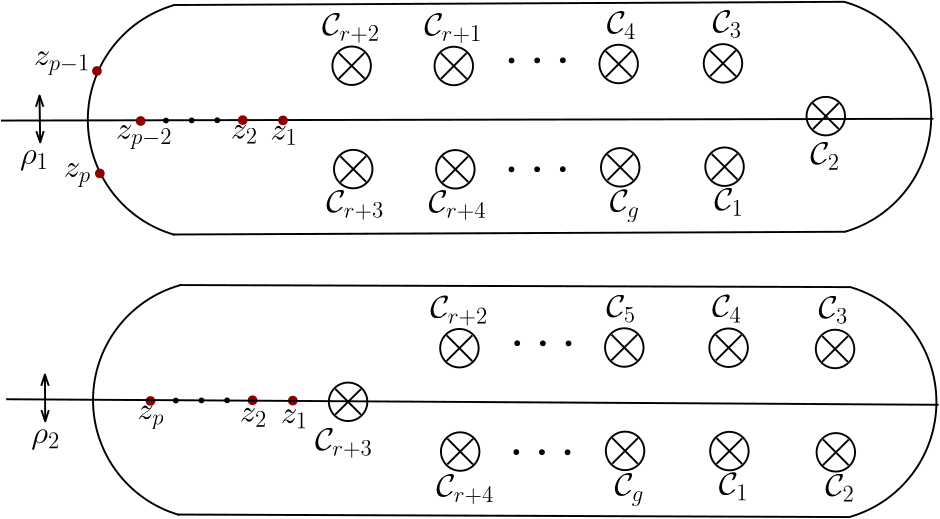}}
\caption{The reflections $\rho_1,\rho_2$ on $N_{g,p}$ for $g=2r+1$.}
\label{nonpunctured_oddrho1-2}
\end{center}
\end{figure}
We work with models of $N_{g,p}$ that are invariant under the reflections $\rho_1$ and $\rho_2$ (depicted in Figures~\ref{nonpunctured_evenrho1-2} and~\ref{nonpunctured_oddrho1-2}). Observe that the mapping class group $\Mod(N_{g,p})$ includes the element $T = \rho_2 \rho_1$.
\begin{figure}[hbt!]
\begin{center}
\scalebox{0.33}{\includegraphics{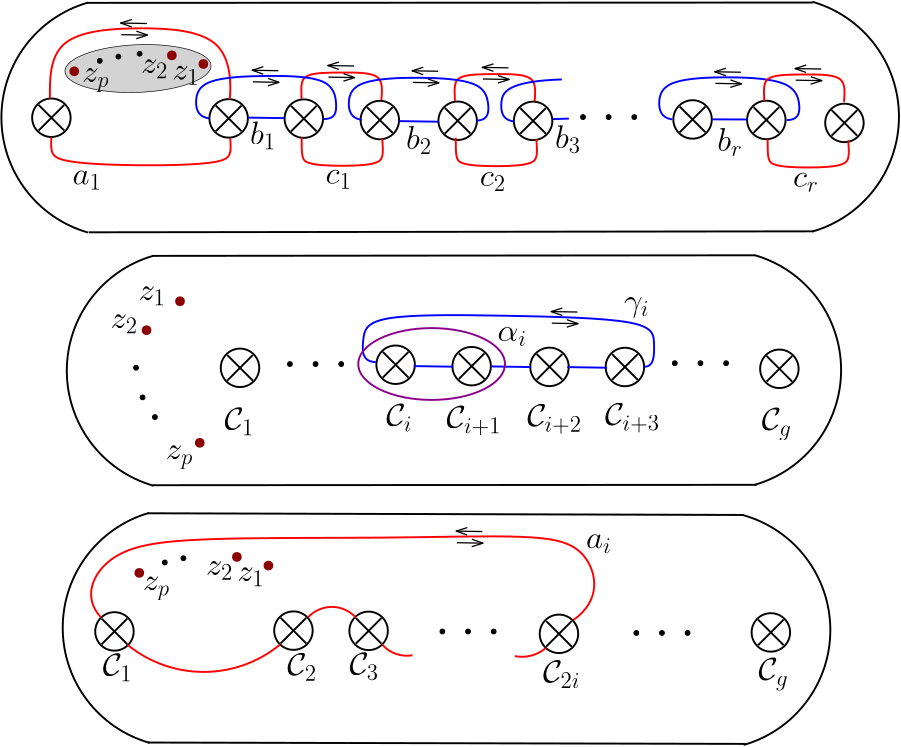}}
\caption{The curves $a_i,b_i,c_i,\gamma_i$ and $\alpha_i$ on  $N_{g,p}$ (note that we do not have $c_r$ in the case of odd genus).}
\label{nonpunctured_curves}
\end{center}
\end{figure}

\begin{figure}[hbt!]
\begin{center}
\scalebox{0.3}{\includegraphics{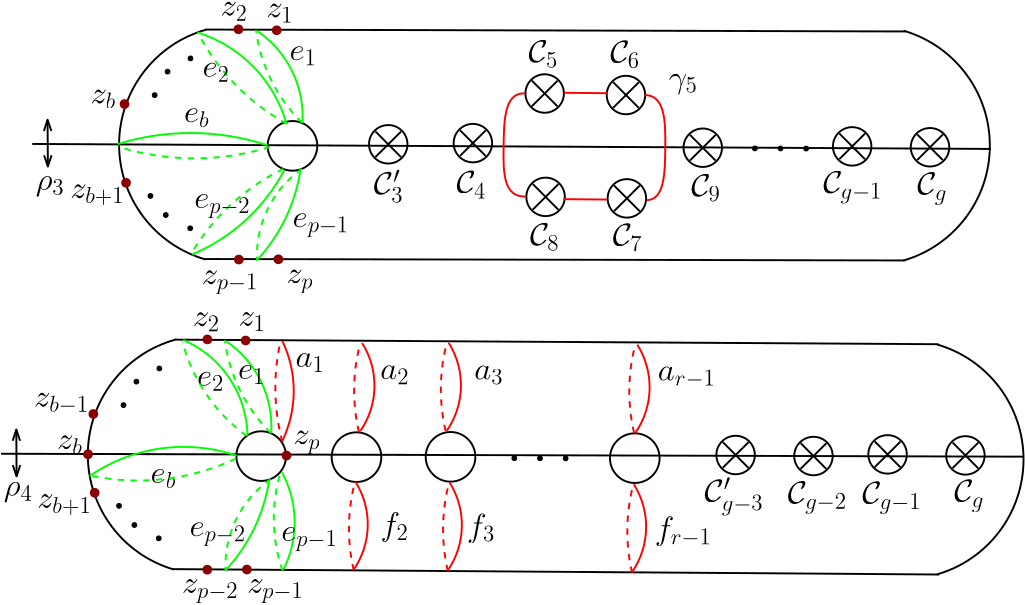}}
\caption{The reflections $\rho_3$ and $\rho_4$ on $N_{g,p}$ if $p=2b$  (If $g$ is odd, the last crosscap is omitted).}
\label{ER34}
\end{center}
\end{figure}

\begin{figure}[hbt!]
\begin{center}
\scalebox{0.3}{\includegraphics{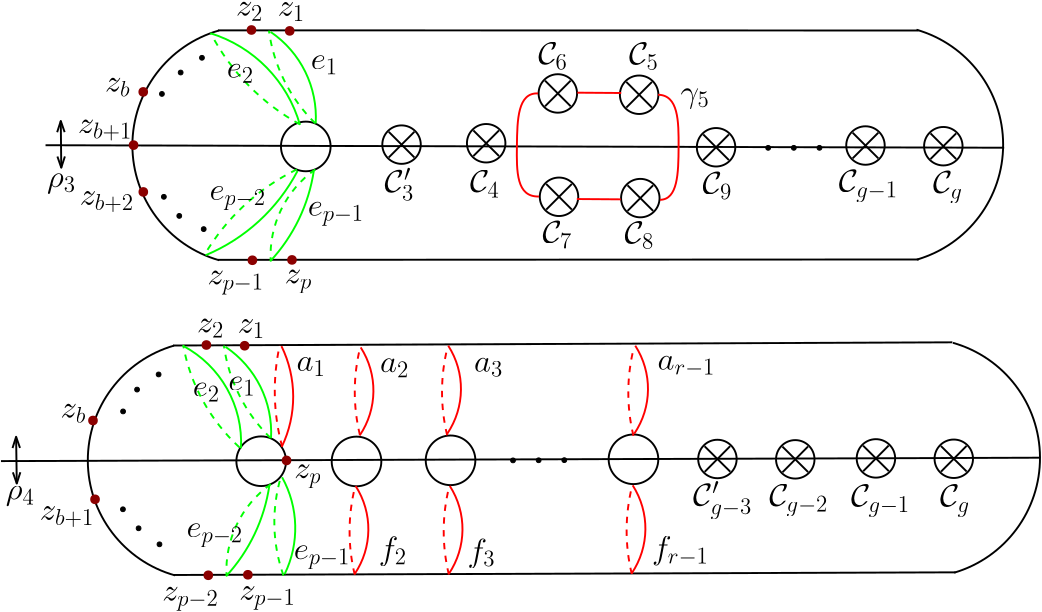}}
\caption{The reflections $\rho_3$ and $\rho_4$ on $N_{g,p}$ if $p=2b+1$  (If $g$ is odd, the last crosscap is omitted).}
\label{OR34}
\end{center}
\end{figure}

Let $u_i$ be the crosscap transposition corresponding to a diffeomorphism of the Klein bottle with one boundary component, where this boundary is the curve $\alpha_i$ as in Figure~\ref{nonpunctured_curves}.
\begin{theorem}\label{nonpunctured_thm_five}
For $g=2r+1\geq 15$ or $g=2r+2\geq 14$, the mapping class group $\Mod(N_{g,p})$ is generated by five elements: \(  T, \rho_3, \rho_4, t_{\gamma_5}t_{b_1}u_{r+5}\) and \(v_{r+3,1}.\)
\end{theorem}
Let $G$ be the subgroup of $\Mod(N_{g,p})$ generated by the elements   $T$, $\rho_3$, $\rho_4$, $t_{\gamma_5}t_{b_1}u_{r+5}$ and $v_{r+3,1}$. The proof follows several steps. First, we show that $G$ contains the Dehn twist generators $t_{a_i}$, $t_{b_i}$, $t_{c_i}$ and $t_{f_i}$ of $\Mod_0(N_{g,p})$ by using conjugation relations and a lantern relation. Here, the diffeomorphism $T$ plays a crucial role due to its action on the curves:
\begin{itemize}
    \item For $i,j = 1,\ldots,r-1$, 
    \(
    T(b_i) = c_i \quad \text{and} \quad T(c_j) = b_{j+1}.
    \)
    \item $T(a_1) = b_1$ and also $T^2(b_r)=a_1$ if $g=2r+1$; and $T(b_r)=c_r$, $T^2(c_r)=a_1$ if $g=2r+2$.
\end{itemize}
Next, the crosscap slide $y=t_{a_1}u_1$ and the puncture slides $v_{g,i}$ (and also $v_{g-1,i}$ if $g$ is even) for $i=1,2,\ldots,p$ are shown to belong to the subgroup $G$. This is achieved by using the actions of $T$ and the reflections $\rho_3$ and $\rho_4$. Additionally, the generators $t_{e_k}$ are contained in $G$ through the actions of $\rho_3$ and $\rho_4$. In the final step of the proof, we show that the elements $\rho_3\rho_4 \in G$ and $\rho_2 \in G$ are mapped to the permutations $(1,2,\ldots,p)$ and $(p, p-1)$ in $Sym_p$, respectively. Since these two permutations generate the entire symmetric group $Sym_p$, the proof is complete.

To construct a generating set consisting entirely of involutions, the elements from the generating set in the preceding theorem are replaced with involutions, ensuring the new set still generates $\Mod(N_{g,p})$. The generating set from Theorem~\ref{invo} is as follows.
\begin{theorem}\label{invo}
    For any genus $g \geq 14$ where $g = 2r + 1 \geq 15$ (odd genus), or $g = 2r + 2 \geq 14$ (even genus), the mapping class group $\Mod(N_{g,p})$ is generated by the following involutions:
    \[
        \begin{cases} 
            \{\rho_1, \rho_2, \rho_3, \rho_4, \rho_3t_{\gamma_5}t_{b_1}u_{r+5}, \rho_2w_{r+3,1}\} & \text{if } g = 2r + 1, \\  
             \{\rho_1, \rho_2, \rho_3, \rho_4, \rho_3t_{\gamma_5}t_{b_1}u_{r+5}, \rho_1w_{r+3,1}\}  & \text{if } g = 2r + 2.
        \end{cases}
    \]
\end{theorem}

\section{Generating sets for the twist subgroup}

The twist subgroup\index{twist subgroup} $\mathcal{T}_g$ is the subgroup generated by all Dehn twists about two-sided simple closed curves in the nonorientable surface $N_g$. The twist subgroup has index $2$ in $\Mod(N_g)$ as shown in~\cite{lickorishnon2}. It plays an important role in the study of mapping class groups of nonorientable surfaces, analogous to the role of the subgroup of Dehn twists in the orientable case. The study of generators for $\mathcal{T}_g$ has seen significant developments over time. Chillingworth~\cite{chillingworth} showed that this group can be generated by finitely many Dehn twists. Stukow~\cite{stukow1} provided an explicit finite presentation using $g+2$ Dehn twist generators. Omori~\cite{omori} later improved this result, showing that $g+1$ Dehn twists suffice if $g\geq 4$. It is not known whether this number is minimal.

When the generators are not required to be Dehn twists, there are some interesting results. Du~\cite{du} showed that for $g=4k+1 \geq 9$, $\mathcal{T}_g$ can be generated by just three elements: two involutions and one element of order $2g$. Recently, Yoshihara~\cite{yoshihara1} was interested in finding a generating set for $\mathcal{T}_g$ consisting of only involutions. He proved that $\mathcal{T}_g$ can be generated by six involutions when $g \geq 14$ and eight involutions when $g \geq 8$. Note that if a group is generated by involutions, its first integral homology group $H_1$ must contain elements of order $2$. This property holds for the twist subgroup $\mathcal{T}_g$ when $g \geq 5$~\cite{stukow}. The authors improved Yoshihara's result. In~\cite{altunoz-pamuk-yildiz4} they proved that the twist subgroup $\mathcal{T}_g $ is generated by:
\begin{compactenum}
\item 4 involutions when either:
  \begin{compactitem}
  \item $g\geq12$ is even, or
  \item $g=4k+1\geq5$.
  \end{compactitem}
\item 5 involutions when either:
  \begin{compactitem}
  \item $g=4k+3\geq11$, or
  \item $g=8,10$.
  \end{compactitem}
\item 6 involutions for $g=6,7$.
\end{compactenum}
Moreover, for $g\geq13$, the authors~\cite{altunoz-pamuk-yildiz3} obtained generating sets for $\mathcal{T}_g$ with three elements: two of which are involutions, and a third element whose order is $g$ or $g-1$ depending on whether $g$ is odd or even. Recently, Leśniak and Szepietowski~\cite{lesniak-szepietowski} showed that for 
$g\neq 4$, $\mathcal{T}_g$ can be generated by three torsion elements, each of whose orders depends on $g$. Since the twist subgroup is not cyclic, a generating set must include at least two elements. Whether $\mathcal{T}_g$ can be generated by two elements was previously unknown. The authors proved that such a generating set exists for odd $g\geq 21$ and even $g\geq 50$~\cite{altunoz-pamuk-yildiz2}. The twist subgroup $\mathcal{T}_g$ is known to be perfect if $g\geq7$~\cite{korkmaz1998,korkmaz2002}. It is interesting to determine for which perfect groups the minimal number of generators is equal to the minimal number of commutator generators. In the orientable case, the mapping class group is a perfect group for $g\geq 3$. It was proved that this group is generated by two commutators if $g\geq 5$, and by three commutators if $g\geq 3$~\cite{baykur-korkmaz2021}. For nonorientable surfaces, analogous results were obtained in~\cite{altunoz-pamuk-yildiz2}, where it was proved that $\mathcal{T}_g$ is generated by
\begin{compactenum}
\item two commutators when either:
  \begin{compactitem}
  \item $g=2r+2\geq50$ or  
  \item $g=4k+1\geq29$.
  \end{compactitem}
\item three commutators when $g=4k+3\geq43$.
\end{compactenum}

In this section, we collect known results about generating sets for the twist subgroup $\mathcal{T}_g$, including proof sketches where the arguments are not too technical. Let $\lbrace x_1, x_2, \ldots, x_{g-1}\rbrace$ be a basis for $H_1(N_g; \mathbb{R})$ such that the curves $x_i$ are pairwise disjoint one-sided curves as shown in Figure~\ref{homology_non}. Since any diffeomorphism $f: N_g \to N_g$ induces a linear map $f_{\ast}:H_1(N_g;\mathbb{R}) \to H_1(N_g;\mathbb{R})$, one can define a determinant homomorphism $D:\Mod(N_g) \to \mathbb{Z}_{2}$ by $D(f)=\textrm{det}(f_{\ast})$. The following lemma from~\cite{lickorishnon1} provides a criterion for determining when a mapping class in $\Mod(N_g)$ belongs to $\mathcal{T}_g$.
\begin{figure}[h]
\begin{center}
\scalebox{0.34}{\includegraphics{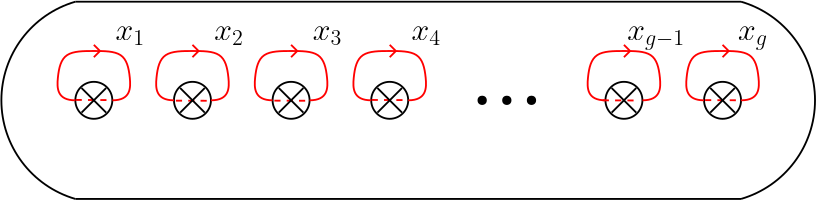}}
\caption{Generators of $H_1(N_g;\mathbb{R})$.}
\label{homology_non}
\end{center}
\end{figure}

\begin{lemma}\label{lemma_twist} 
A mapping class $f\in  \Mod(N_g)$ satisfies:
\begin{itemize}
\item $D(f) = 1$ if and only if $f\in \mathcal{T}_g$,
\item $D(f) = -1$ if and only if $f \not \in \mathcal{T}_g$.
\end{itemize}
\end{lemma}

Consider the surface $N_g$ as depicted in Figure~\ref{NG}. Omori's Dehn twist generators for the twist subgroup $\mathcal{T}_g$ are given as follows, where the curve $\epsilon$ is as in Figure~\ref{generators_twist} (note that the curve $c_{r}$ does not exist when $g$ is odd).
\begin{theorem}\cite{omori}\label{twistthm1}
The twist subgroup $\mathcal{T}_g$ is generated by the following $(g+1)$ Dehn twists:
\begin{itemize}
\item $t_{a_1},t_{a_2},t_{b_1},\ldots, t_{b_r}$, $t_{c_1},\ldots, t_{c_{r-1}}$ and $t_{\epsilon}$ if $g=2r+1$, and 
\item $t_{a_1},t_{a_2},t_{b_1},\ldots, t_{b_r}$, $t_{c_1},\ldots, t_{c_{r-1}},t_{c_{r}} $ and $t_{\epsilon}$ if $g=2r+2$.
\end{itemize}
\end{theorem}

\subsection{Minimal number of generators for the twist subgroup}
 In this section, we focus on generating sets for the twist subgroup with minimal cardinality. It is known that such generating sets must contain at least two elements. A generating set containing two elements was obtained by the authors. They obtained a two-element generating set for the twist subgroup $\mathcal{T}_g$ if $g\geq 21$ is odd and $g\geq 50$ is even~\cite{altunoz-pamuk-yildiz2}. We now provide an explicit description of every element in the generating set mentioned in this result.
\begin{figure}[h]
\begin{center}
\scalebox{0.55}{\includegraphics{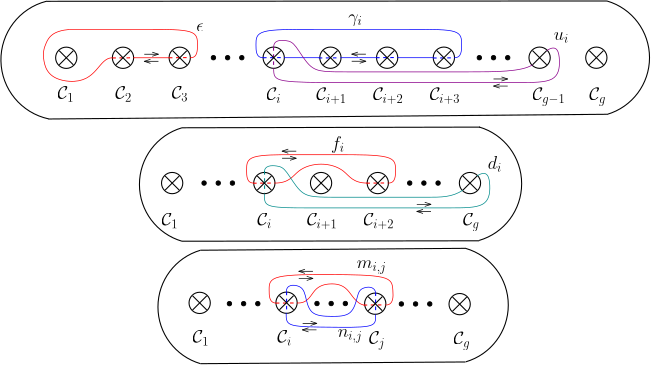}}
\caption{The curves $\epsilon, \gamma_i, u_i, f_i, d_i, m_{i,j}$ and $n_{i,j}$ on $N_g$.}
\label{generators_twist}
\end{center}
\end{figure}

We begin by defining a torsion element in our generating set: the rotation map $T$. This element's geometric interpretation depends on the parity of the genus. When $g$ is odd, $T$ is the rotation by $2\pi/g$ that cyclically permutes the crosscaps (see Figure~\ref{Todd_twist}); it has order $g$. When $g$ is even, $T$ is the rotation by $2\pi/(g-1)$ that cyclically permutes $g-1$ of the crosscaps while fixing one (see Figure~\ref{Teven_twist}); it has order $g-1$. In both configurations, the rotation $T$ satisfies $D(T)=1$, ensuring it belongs to $\mathcal{T}_g$.
 \begin{figure}[h]
\begin{center}
\scalebox{0.35}{\includegraphics{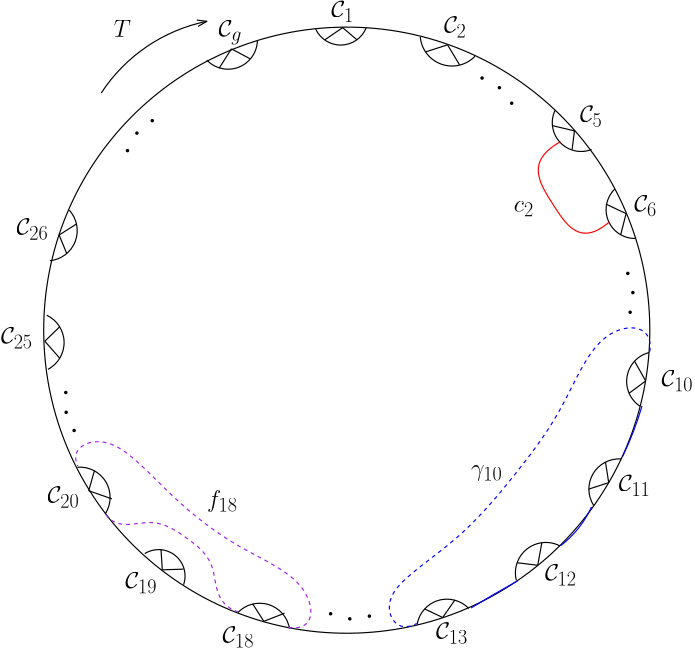}}
\caption{The rotation $T$ for $g=2r+1$.}
\label{Todd_twist}
\end{center}
\end{figure}

 \begin{figure}[h]
\begin{center}
\scalebox{0.35}{\includegraphics{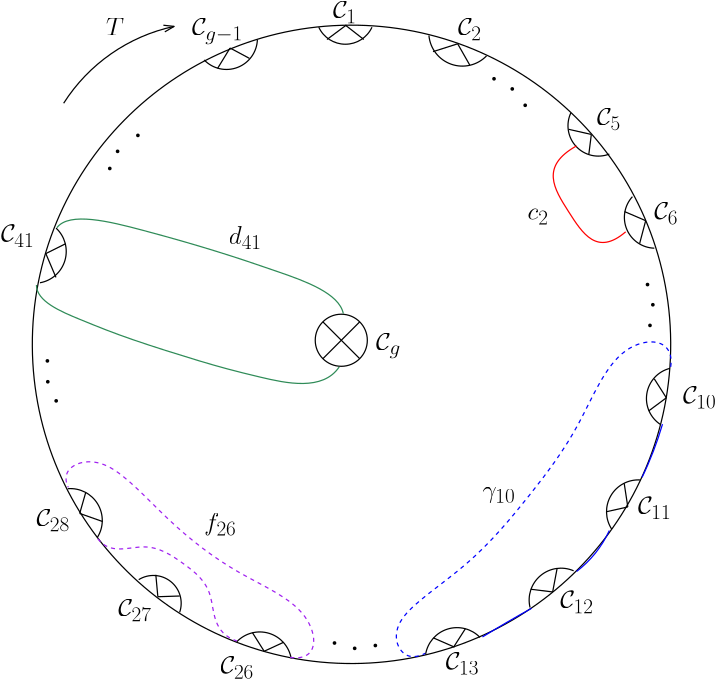}}
\caption{The rotation $T$ for $g=2r+2$.}
\label{Teven_twist}
\end{center}
\end{figure}
Now, we can give the generators for $\mathcal{T}_g$ from ~\cite{altunoz-pamuk-yildiz2}:
\begin{theorem}\label{theorem_twist_two}
The twist subgroup $\mathcal{T}_g$ is generated by the following two elements:
\begin{equation*}
 \begin{cases}
   T, t_{\gamma_{10}}t_{c_2}^{-1}t_{f_{18}} & \text {if } g=2r+1\geq 21,\\
        T, t_{\gamma_{10}}t_{c_2}^{-1}t_{f_{26}}t_{d_{41}}^{-1}
         & \text {if } g=2r+2\geq 50,
    \end{cases}    
\end{equation*}  
 where all curves are shown in Figures~\ref{Todd_twist} and ~\ref{Teven_twist}.
\end{theorem}

For smaller genera ($g \geq 8$), the number of generators increases. The generating sets are as follows.
\begin{theorem}\label{theorem_twist_three}The twist subgroup $\mathcal{T}_g$ is generated by the following three elements:
  \begin{equation*}
 \begin{cases}
   T, t_{d_{g-1}}t_{a_2}^{-1}, t_{f_{1}}t_{b_2}^{-1} & \text {if } g=2r+2\geq 8,\\
        T, t_{a_{1}}t_{a_2}^{-1}, t_{f_{1}}t_{b_{2}}^{-1}
         & \text {if } g=2r+1\geq 9,
    \end{cases}       
\end{equation*} 
where the curves are as in Figures~\ref{NG} and ~\ref{generators_twist}.
\end{theorem}

 \subsection{Torsion generators for the twist subgroup}
 For orientable surfaces, generating sets of the mapping class group by torsion elements have been extensively studied. Analogously, a natural question is whether the twist subgroup $\mathcal{T}_g$ can be generated by torsion elements. This problem has seen significant recent advances. We review known results across different genera, focusing on minimal generator numbers and their orders.

The twist subgroup $\mathcal{T}_2$ is isomorphic to $\mathbb{Z}_2$~\cite{lickorishnon2}. The group $H_1(\mathcal{T}_4) \cong \mathbb{Z}_2\times \mathbb{Z}$, so ${\mathcal{T}_4}$ is not generated by torsion elements. For $g=4k+1 \geq 9$, Du~\cite{du} gave a generating set for $\mathcal{T}_g$ consisting of three torsion elements: two involutions and one element of order $2g$. Later, the authors~\cite{altunoz-pamuk-yildiz3} showed that when $g\geq13$, $\mathcal{T}_g$ can be generated by two involutions together with a single additional element whose order is equal to the genus $g$ when odd and $g-1$ when even. More recently, Leśniak and Szepietowski~\cite{lesniak-szepietowski} showed that for $g\neq 4$, $\mathcal{T}_g$ is generated by three torsion elements. We will discuss this set of generators below. 

We recall the definitions of the following elements from Section~\ref{torsionnonorientable}.
\begin{itemize}
    \item $s = 
    \begin{cases}
        t_{a_1}t_{b_1}t_{c_1}t_{b_2}t_{c_2} \cdots t_{b_r}t_{c_r} & \text{if } g = 2r + 2, \\
        t_{a_1}t_{b_1}t_{c_1}t_{b_2}t_{c_2} \cdots t_{c_{r-1}}t_{b_r} & \text{if } g = 2r + 1.
    \end{cases}$
    
    \item $s' = 
    \begin{cases}
        t_{a_1}^2t_{b_1}t_{c_1}t_{b_2}t_{c_2}\cdots t_{b_r}t_{c_r} & \text{if } g=2r+2, \\
        t_{a_1}t_{b_1}t_{c_1}t_{b_2}t_{c_2}\cdots t_{c_{r-1}}t_{b_r} & \text{if } g=2r+1.
 \end{cases}$
    
    \item $y = 
    \begin{cases}
        t_{b_r}u_{g-1} & \text{if } g = 2r + 1, \\
        t_{c_r}u_{g-1} & \text{if } g = 2r + 2.
    \end{cases}$
\end{itemize}

Recall also that the orders of $s$ and $s'$ depend on the parity of $g$: when $g$ is even, $s$ has order $g$ and $s'$ has order $g-1$, while for odd $g$, their orders are $2g$ and $2(g-1)$, respectively. By definition, $s$ and $s'$ are contained in the twist subgroup $\mathcal{T}_g$. For $g=3$, the twist subgroup is generated by the Dehn twists $t_{a_1}$ and $t_{b_1}$. It is easy to see that it can also be generated by $t_{a_1}t_{b_1}$ and $t_{a_1}^2t_{b_1}$, whose orders are six and four, respectively. By adapting the arguments used in the proof of Theorem~\ref{torsionnonthm} and substituting the elements $t$ and $t'$ with $s$, one can deduce that the twist subgroup $\mathcal{T}_5$ is generated by the torsion elements $s$, $st_{a_2}$ and $y^{-1}sy$. For $g\geq6$, $\mathcal{T}_g$ is generated by the torsion elements $s$, $s'$ and $xsx^{-1}$. Thus, we can state the following theorem: 

\begin{theorem}
    For $g\neq 4$, the twist subgroup $\mathcal{T}_g$ is generated by three torsion elements.
\end{theorem}

For sufficiently large genus $g$, the twist subgroup $\mathcal{T}_g$ is generated by three conjugate torsion elements. The proof adapts Lanier's methods~\cite{lanier2018} and parallels Theorem~\ref{lesniak-szepietowski-non} (for a complete proof, see~\cite{lesniak-szepietowski}).
\begin{theorem}\label{lesniak-szepietowski-twist}
Let $k\geq 12$. For a nonorientable surface of genus $g=pk+2q(k-1)$ or $g=pk+2q(k-1)+1$ with $q\in \mathbb Z_{> 0}$, $\mathcal{T}_g$ is generated by three conjugate elements of order $k$ if either $k$ is odd or $p$ is even.
\end{theorem}

 \subsection{Involution generators for the twist subgroup}
For orientable surfaces, the mapping class group has been extensively studied, yielding deep insights into its involution-based generating sets. It is known that the mapping class group of an orientable surface can be generated by only three involutions—the minimal number possible for such a generating set—provided the genus is six or larger~\cite{korkmaz2020,yildiz2020}. In contrast, less is known about the mapping class groups of nonorientable surfaces, particularly regarding generating sets with specific properties. In this section, we focus on generating the twist subgroup $\mathcal{T}_g$ by only involutions. It is known that $\mathcal{T}_g$ cannot be generated by two involutions (as any such group is a quotient of a dihedral group), so it is required that any generating set consisting of only involutions must contain at least three such elements. Yoshihara~\cite{yoshihara1} first studied generating $\mathcal{T}_g$ using only involutions. His results demonstrated that $\mathcal{T}_g$ can be generated by six involutions when $g\geq 14$, while for surfaces with $g\geq 8$, eight involutions suffice to generate the group. Later, the authors~\cite{altunoz-pamuk-yildiz4} improved Yoshihara's result by presenting generating sets for $\mathcal{T}_g$ that require a smaller number of involutions. First, we state Yoshihara's generators.
\begin{figure}[h]
\begin{center}
\scalebox{0.33}{\includegraphics{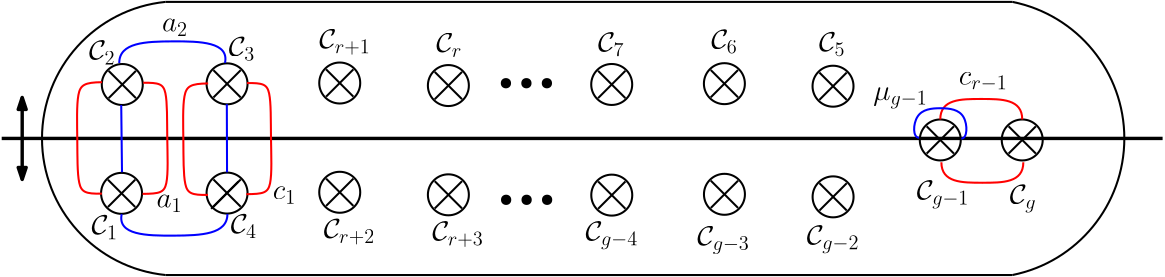}}
\caption{The reflection $\sigma$ for $g=2r$.}
\label{sigmayoshihara}
\end{center}
\end{figure}
\begin{figure}[h]
\begin{center}
\scalebox{0.35}{\includegraphics{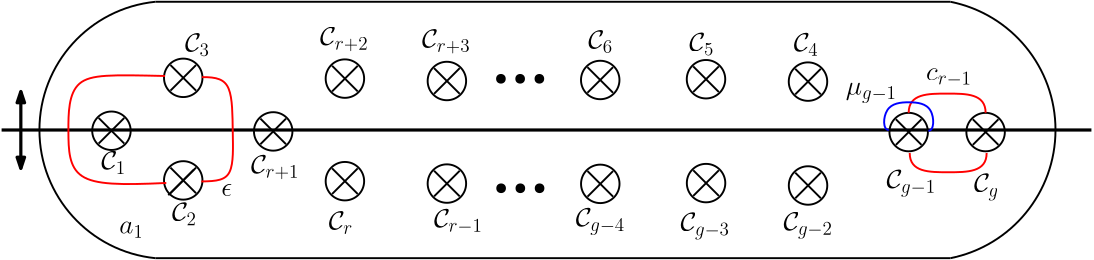}}
\caption{The reflection $\tau$ for $g=2r$.}
\label{tauyoshihara}
\end{center}
\end{figure}
\begin{figure}[h]
\begin{center}
\scalebox{0.45}{\includegraphics{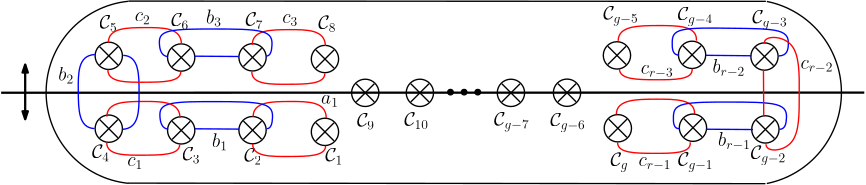}}
\caption{The reflection $\upsilon$ for $g=2r$.}
\label{vyoshihara}
\end{center}
\end{figure}
Let $g=2r\geq 14$. We work with the models for $N_g$ shown in Figures~\ref{sigmayoshihara}, ~\ref{tauyoshihara} and ~\ref{vyoshihara}. It is known that homeomorphisms exist between these models. Note that the surface $N_g$ is invariant under the involutions $\sigma$, $\tau$ and $\upsilon$, which are reflections across the corresponding planes as shown in Figures~\ref{sigmayoshihara}, ~\ref{tauyoshihara} and ~\ref{vyoshihara}. For $g=2r+1\geq 15$, the crosscaps can be distributed similarly to the case of $g=2r\geq 14$ in order to define the reflections $\sigma$, $\tau$ and $\upsilon$ (see~\cite[Figures 11, 12 and 13]{yoshihara1}). By an easy calculation, the reflection $\sigma$ satisfies: 
\begin{itemize}
    \item $D(\sigma)=-1$ if r is odd,
     \item $D(\sigma)=1$ if r is even,
     \item $D(\sigma)=-1$ if g is odd.
\end{itemize}
To ensure membership in $\mathcal{T}_g$ (where $D(\sigma)=1$), we compose the reflection $\sigma$ with the crosscap slide $Y_{\mu_{g-1},c_{r-1}}$, which satisfies $D(Y_{\mu_{g-1},c_{r-1}})=-1$. In this case, abusing notation, we keep writing $\sigma$ instead of $\sigma Y_{\mu_{g-1},c_{r-1}}$. Similarly the reflection $\tau$ satisfies:  
\begin{itemize}
    \item $D(\tau)=1$ if r is odd,
     \item $D(\tau)=-1$ if r is even,
     \item $D(\tau)=-1$ if g is even.
\end{itemize}
When $\tau$ fails to lie in  $\mathcal{T}_g$, we adjust it by composing it with the crosscap slide $Y_{\mu_{g-1},c_{r-1}}$ so that it falls into $\mathcal{T}_g$. We continue to write $\tau$ for the composed mapping class. The third reflection $\upsilon$ satisfies $D(\upsilon)=1$, which means that $\upsilon \in \mathcal{T}_g$. We define three additional involutions:
\[
\rho_1 = \sigma t_{a_1}, \rho_2 = \sigma t_{b} \textrm{ and }
\rho_3 = \tau t_{\epsilon}.
\]
It can be verified that $\rho_1$, $\rho_2$ and $\rho_3$ are involution elements in $\mathcal{T}_g$. These elements are used to obtain Omori's generators $t_{a_1}$, $t_{b}$, and $t_{\epsilon}$ (Theorem~\ref{twistthm1}). Let $G = \langle \sigma, \tau, v, \rho_1, \rho_2, \rho_3 \rangle$. All Dehn twists in Omori's generating set (Theorem~\ref{twistthm1}) lie in $G$: 
\begin{itemize}
    \item $t_{a_1} = \sigma\rho_1\in G$,
    \item $t_b = \sigma\rho_2\in G$ and
    \item $t_{\epsilon} = \tau\rho_3\in G$.
\end{itemize}
The remaining Dehn twist generators are obtained through the action of $\sigma, \tau, v$ on $a_1$. Thus the subgroup $G$ contains all required generators of the twist subgroup $\mathcal{T}_g$. For other Dehn twist generators, the action of $\sigma, \tau, v$ maps $a_1$ to other Dehn twist generators. Thus $G = \mathcal{T}_g$. We state Yoshihara's result:
\begin{theorem}~\cite{yoshihara1}
  For $g\geq 16$ and $g=14$, the twist subgroup $\mathcal{T}_g$ can be generated by six involutions.
\end{theorem}

For $8 \leq g \leq 16$, we replace $\upsilon$ with three additional involutions (see for details~\cite[Theorem 4.1]{yoshihara1}) (constructed similarly via reflections and crosscap slides), yielding a generating set consisting of eight involutions.
\begin{theorem}~\cite{yoshihara1}
  For $g\geq 8$, the twist subgroup $\mathcal{T}_g$ can be generated by eight involutions.
\end{theorem}

We now state the involution generators for the twist subgroup $\mathcal{T}_g$ given by Altunöz, Pamuk and Yildiz~\cite{altunoz-pamuk-yildiz4}. For $g=2r+2$, we consider the models in Figures~\ref{involutiontwisteven_tay} and~\ref{sigma_tay}.
\begin{figure}[h]
\begin{center}
\scalebox{0.45}{\includegraphics{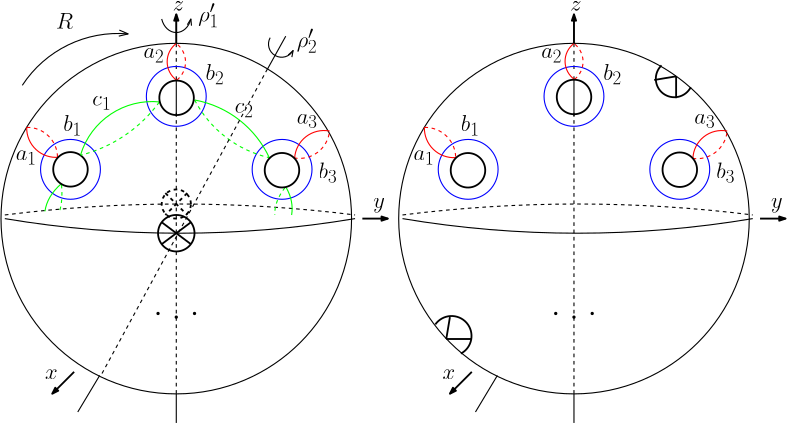}}
\caption{The models for $N_g$ if $g=2r+2$.}
\label{involutiontwisteven_tay}
\end{center}
\end{figure} 
\begin{figure}[h]
\begin{center}
\scalebox{0.35}{\includegraphics{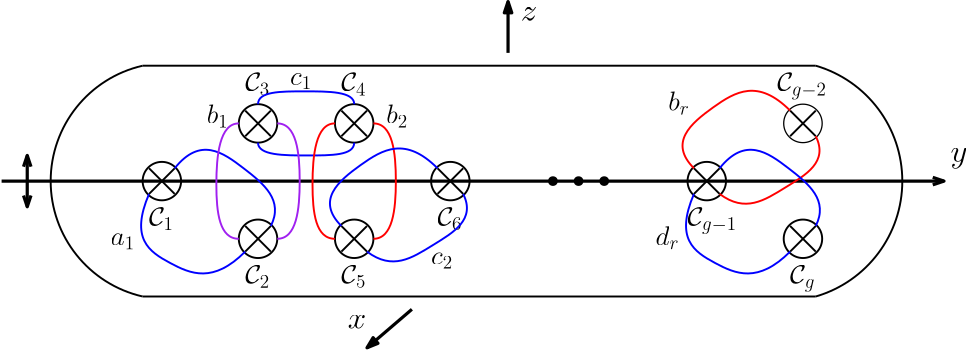}}
\caption{The involution $\sigma$ for $g=2r+2$.}
\label{sigma_tay}
\end{center}
\end{figure}For $g=2r+2$, we consider the models in Figures~\ref{involutiontwisteven_tay} and~\ref{sigma_tay}. In the models shown in Figure~\ref{involutiontwisteven_tay}, we embed a genus $r$ orientable surface with two disks removed in $\mathbb{R}^{3}$, arranging each genus in a circular position with the second genus on the $+z$-axis. The rotation by
$\frac{2\pi}{r}$ about the $x$-axis cyclically permutes the curves, mapping $b_i$ to $b_{i+1}$ for $i=1,...,r-1$ and sending $b_r$ back to $b_1$. The right-hand model results from applying a crosscap-sliding diffeomorphism, say $\phi$, to the left-hand model. Let $\tau^{\prime}$ denote the blackboard reflection of $N_g$ for the left model in Figure~\ref{involutiontwisteven_tay}. Note that $D(\tau^{\prime})=-1$ when $r$ is even. For odd $r$, we use the reflection $\tau^{\prime}$, while for even $r$ we use the conjugate reflection $\phi\tau^{\prime}\phi^{-1}$. The surface $N_g$ remains invariant under both reflections. Abusing notation, we will continue to write $\tau^{\prime}$ instead of $\phi\tau^{\prime}\phi^{-1}$ when $r$ is even. Now define two involutions $\rho_1$ and $\rho_2$: consider $\rho_{1}^{\prime}$ ($\pi$-rotation about the $z$-axis) and $\rho_{2}^{\prime}$ ($\pi$-rotation about the line $z = \tan(\pi/r)y$, $x=0$), as shown in Figure~\ref{involutiontwisteven_tay}. Both satisfy $D(\rho_{1}^{\prime}) = D(\rho_{2}^{\prime}) = -1$, which implies that they are not in $\mathcal{T}_g$. We define $\rho_1=\rho_{1}^{\prime}\tau$ and $\rho_2=\rho_{2}^{\prime}\tau$. The resulting elements $\rho_{1}$ and $\rho_{2}$ are involutions that lie in the twist subgroup $\mathcal{T}_g$. Observe that the rotation $R$ is equal to $\rho_1\rho_2$.

For the last involution element in $\mathcal{T}_g$, we consider $N_g$ where $g$ crosscaps are 
distributed as in Figure~\ref{sigma_tay}. For $g=2r+2$ and $r\geq3$, there is a 
reflection, $\sigma$, of the surface $N_g$ shown in Figure~\ref{sigma_tay} in the $xy$ plane. The involution $\sigma$ is in $\mathcal{T}_g$ if $g$ is even since it satisfies $D(\sigma)=1$.
\begin{theorem}\label{theorem_involution_twist_even_tay}
 The twist subgroup $\mathcal{T}_g$ can be generated by the following involution elements:
\begin{equation*}
 \begin{cases}
      \rho_1, \rho_2, \rho_1 t_{a_2} t_{c_{\frac{r}{2}}}t_{b_{\frac{r+4}{2}}}t_{c_{\frac{r+6}{2}}}, \sigma & \text {if } g=2r+2, r\geq 6 \text{ and even,}\\
        \rho_1, \rho_2, \rho_1 t_{a_1} t_{b_2}t_{c_{\frac{r+3}{2}}}t_{a_3}, \sigma  & \text {if } g=2r+2, r\geq 7 \text{ and odd,}\\
        \rho_1, \rho_2, \rho_1 t_{a_1} t_{b_2}t_{c_{4}}t_{a_3}, \sigma  & \text {if } g=12.
 \end{cases}       
\end{equation*}   
\end{theorem}

(Here, a straightforward computation verifies that the third element in each generating set is indeed an involution.) For $g=6,8$ and $10$, we use more involutions to generate $\mathcal{T}_g$ (see~\cite[Theorems 3.6, 3.7 and 3.8]{altunoz-pamuk-yildiz4} for the additional involution generators). In these cases, while five involutions generate $\mathcal{T}_{10}$ and $\mathcal{T}_8$, the twist subgroup $\mathcal{T}_{6}$ requires six such generators. Let us now state the involution generators for $g=2r+1$.

For $g=4k+1$, we consider two models for $N_g$: the left model in Figure~\ref{invoddtwist} and the model in Figure~\ref{twistoddrho12}. In the first model, the crosscaps $\mathcal{C}_i$ are in a circular position with $\mathcal{C}_2$ on the positive $z$-axis. The rotation $T$ by $2\pi/g$ about the $x$-axis cyclically permutes the crosscaps via $\mathcal{C}_i \mapsto \mathcal{C}_{i+1}$ modulo $g$. In the second model, we embed a genus $r$ surface (minus one disk) in $\mathbb{R}^{3}$ with handles in a circular position, and the crosscap $\mathcal{C}_2$ is on the $+z$-axis. The rotation $R$ by
$\frac{2\pi}{r}$ about the $x$-axis maps the curve $b_i$ to $b_{i+1}$ modulo $r$. Define two involutions $\rho_{1}$ and $\rho_{2}$, where $\rho_{1}$ is the reflection in the $xz$-plane and $\rho_{2}$ is the reflection in the plane $z=\tan(\frac{\pi}{r})y$ as depicted in Figure~\ref{twistoddrho12}. The mapping classes $\rho_{1}$ and $\rho_{2}$ satisfy $D(\rho_{1})=D(\rho_{2})=1$ if $g=4k+1$. In this case, they are contained in $\mathcal{T}_g$, and the rotation $R$ is equal to $\rho_1\rho_2$.

For $g=4k+3$, we consider the right model in Figure~\ref{invoddtwist}. In this model, each crosscap $\mathcal{C}_i$ for $i=1,\ldots,g-2$ is in a circular position with the second crosscap $\mathcal{C}_2$ on the $+z$-axis. The rotation $T$ by $\frac{2\pi}{g-2}$ about the $x$-axis maps the crosscap $\mathcal{C}_i$ to $\mathcal{C}_{i+1}$ for $i=1,\ldots,g-3$. The crosscap  $\mathcal{C}_{g-1}$ is on the $+x$-axis, and $\mathcal{C}_g$ is obtained by rotating  $\mathcal{C}_{g-1}$ by $\pi$ about the $+z$-axis.
\begin{figure}[h]
\begin{center}
\scalebox{0.35}{\includegraphics{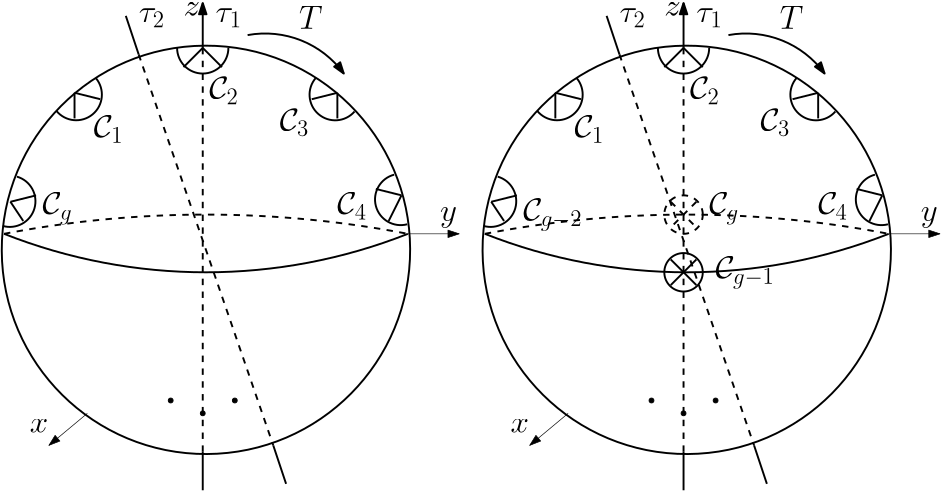}}
\caption{The involutions $\tau_1$ and $\tau_2$ for $g=2r+1$.}
\label{invoddtwist}
\end{center}
\end{figure} 

\begin{figure}[h]
\begin{center}
\scalebox{0.35}{\includegraphics{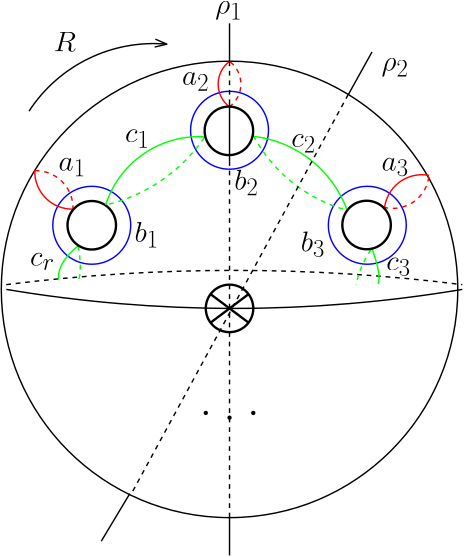}}
\caption{The involutions $\rho_1$ and $\rho_2$ for $g=2r+1$.}
\label{twistoddrho12}
\end{center}
\end{figure} 
\begin{figure}[h]
\begin{center}
\scalebox{0.35}{\includegraphics{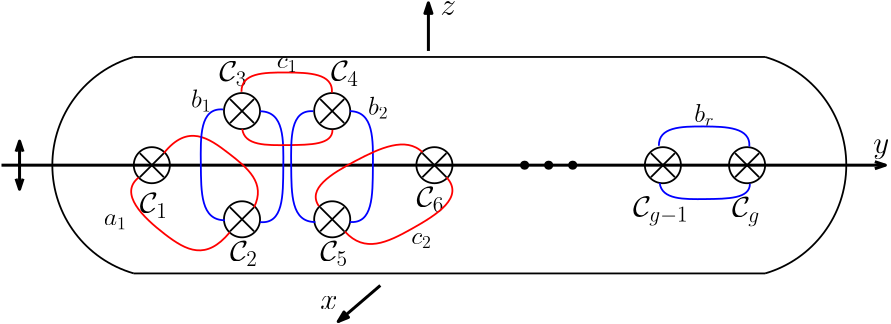}}
\caption{The involution $\beta$ for $g=2r+1$.}
\label{beta_twist}
\end{center}
\end{figure} 
\begin{figure}[h]
\begin{center}
\scalebox{0.35}{\includegraphics{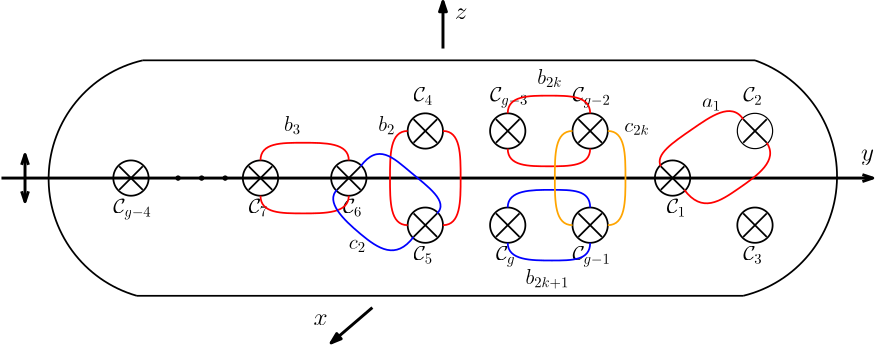}}
\caption{The involution $\mu$ for $g=4k+3$.}
\label{mu_twist}
\end{center}
\end{figure}

For \( g = 4k + 3 \), consider the right-hand model shown in Figure~\ref{invoddtwist}. Define two involutions $\tau_1$ and 
$\tau_2$, where $\tau_1$ is the reflection about the $z$-axis and $\tau_2$ is the reflection in the plane  $z=\tan(\frac{\pi}{r})y$. We have $D(\tau_1)=D(\tau_2)=1$ if $g = 4k + 3$, and so $\tau_1$ and $\tau_2$ are in $\mathcal{T}_g$. Let us consider $N_g$ where $g$ crosscaps are distributed as in Figure~\ref{beta_twist} to define the involution $\beta$, which is the reflection in the $xy$ plane. The involution $\beta$ satisfies $D(\beta)=1$, and hence the mapping class $\beta$ is an element of $\mathcal{T}_g$. Let $\mu$ be another reflection in the $xy$ plane defined on the model for $N_g$ in Figure~\ref{mu_twist}. It can be verified that $D(\mu)=1$ for $k\geq2$. Thus $\mu$ is in $\mathcal{T}_g$ for $k\geq2$.
\begin{theorem}\label{theorem_involution_twist_odd_tay}
The twist subgroup $\mathcal{T}_g$ can be generated by the following involution elements:
\begin{equation*}
 \begin{cases}
      \rho_1,\rho_2,\rho_1t_{a_2}t_{c_{\frac{r}{2}}}t_{b_{\frac{r+4}{2}}}t_{c_{\frac{r+6}{2}}},\beta & \text {if } g=4k+1 \text{ and } k\geq3,\\
        \tau_1,\tau_2, \tau_1\tau_2\tau_1A_2,\tau_2A_1,\beta  & \text {if } g=4k+1 \text{ and } k\geq1,\\
        \tau_1,\tau_2, \tau_1\tau_2\tau_1A_2,\tau_2A_1,\mu  & \text {if } g=4k+3 \text{ and } k\geq1.
 \end{cases}       
\end{equation*} 
\end{theorem}
For $g=5,7$ and $9$, using additional (or different) involution elements, it can be shown that $\mathcal{T}_5$ and $\mathcal{T}_9$ can be generated by four involutions while $\mathcal{T}_7$ requires six involutions.
\subsection{Commutator generators for the twist subgroup}
The twist subgroup $\mathcal{T}_g$ is perfect\index{perfect group}—that is, it is equal to its own commutator subgroup\index{commutator subgroup}—for genus $g \geq 7$ ~\cite{korkmaz1998,korkmaz2002}. A natural question concerning perfect groups is whether the minimal number of generators is equal to the minimal number of commutator generators\index{generating set!commutator}. In the case of the mapping class group of a closed orientable surface of genus $g$—which is also perfect for $g \geq 3$—Baykur and Korkmaz~\cite{baykur-korkmaz2021} proved that the mapping class group can be generated by two commutators if $g \geq 5$ and by three commutators if $g \geq 3$. This motivates analogous investigations for the twist subgroup $\mathcal{T}_g$. In~\cite{altunoz-pamuk-yildiz2}, the authors examined this question for $\mathcal{T}_g$; we discuss their results in this section. The proof proceeds by cases based on the parity of $g$ and its value modulo $4$.
\begin{figure}[h]
\begin{center}
\scalebox{0.35}{\includegraphics{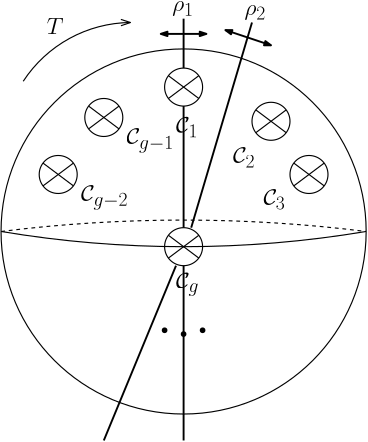}}
\caption{The reflections $\rho_1$ and $\rho_2$ on the surface $N_g$ for $g=4k$.}
\label{rho4k}
\end{center}
\end{figure}

\begin{figure}[h]
\begin{center}
\scalebox{0.21}{\includegraphics{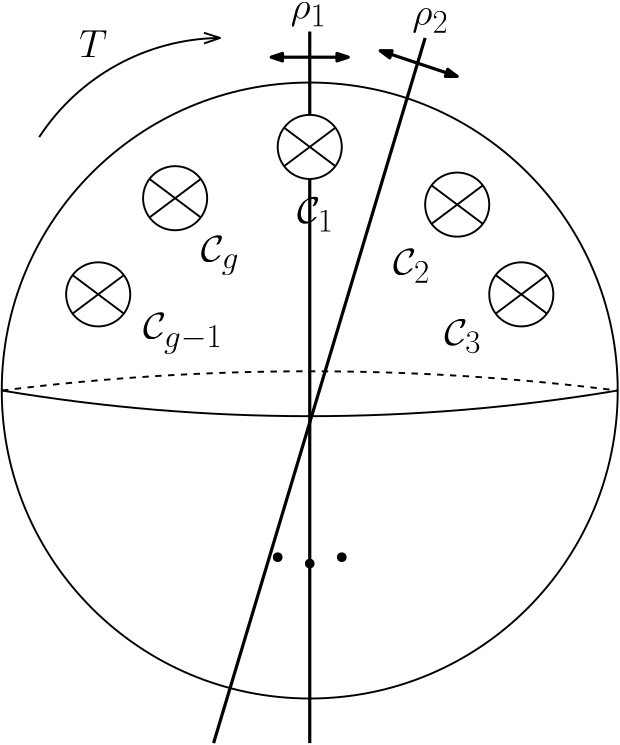}}
\caption{The reflections $\rho_1$ and $\rho_2$ on the surface $N_g$ for $g=4k+1$.}
\label{rho4k+1}
\end{center}
\end{figure}
\begin{figure}[h]
\begin{center}
\scalebox{0.25}{\includegraphics{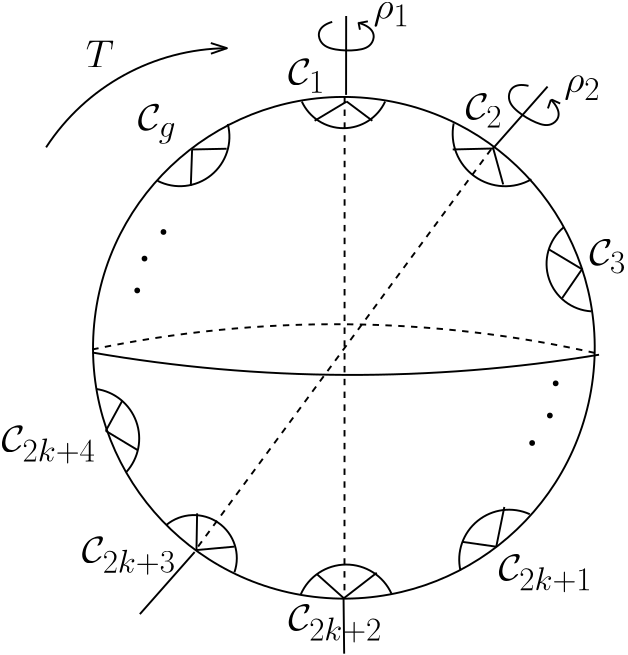}}
\caption{The rotations $\rho_1$ and $\rho_2$ on the surface $N_g$ for $g=4k+2$.}
\label{rho4k+2}
\end{center}
\end{figure}
\begin{figure}[h]
\begin{center}
\scalebox{0.29}{\includegraphics{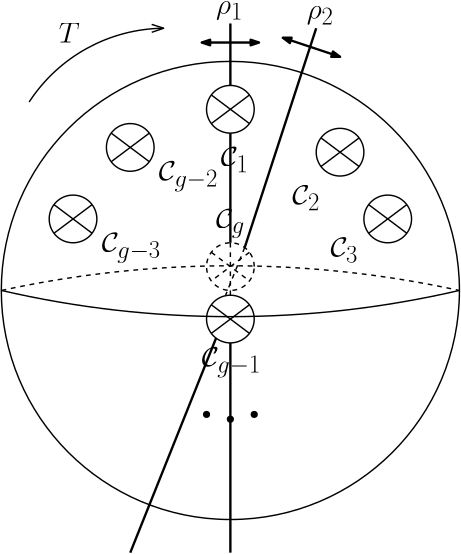}}
\caption{The reflections $\rho_1$ and $\rho_2$ on the surface $N_g$ for $g=4k+3$.}
\label{rho4k+3}
\end{center}
\end{figure}
We define the mapping classes $\rho_1$ and $\rho_2$ differently for each congruence class of $g$ modulo $4$ so that the rotation $T$ is expressed as $T=\rho_2\rho_1$:
\begin{itemize}
     \item \textbf{Case $g=4k$}: 
    The reflections $\rho_1,\rho_2$ shown in Figure~\ref{rho4k} lie in $\mathcal{T}_g$.
 \item \textbf{Case $g=4k+1$}: 
    The reflections $\rho_1,\rho_2$ in Figure~\ref{rho4k+1} belong to $\mathcal{T}_g$.
 \item \textbf{Case $g=4k+2$}: 
    $\rho_1,\rho_2$ are $\pi$-rotations (Figure~\ref{rho4k+2}).
 \item \textbf{Case $g=4k+3$}: 
    We again use reflections $\rho_1,\rho_2$ as in Figure~\ref{rho4k+3}.
\end{itemize}

For each case, the rotation $T = \rho_2\rho_1$ can be expressed as a commutator in $\mathcal{T}_g$, as shown in the following proposition (for details see the proof in ~\cite[Proposition 4.1]{altunoz-pamuk-yildiz2}).
\begin{proposition}\label{prop_comm}
The rotation $T \in \mathcal{T}_g$ can be written as a commutator as follows:
\begin{itemize}
    \item $T = [T^{2k}, \rho_1]$ for $g=4k$,
\item $T = [T^{2k+1}, \rho_1]$ for $g=4k+1$ or $g=4k+3$, 
\item $T = [T^{k+1}, \rho_1]$, for $g=4k+2$,
\end{itemize}
where $k \geq 1$ in each case.
\end{proposition}
We are now ready to give the main result of this section. We work with generating sets for $\mathcal{T}_g$ in which every generator can be represented as a single commutator. Notably, the rotation $T$ in each set can be expressed as a commutator by the proposition above.
\begin{itemize}
\item$g=4k\geq50$: By Theorem~\ref{theorem_twist_two}, $\mathcal{T}_g$ is generated by $T$ and $t_{\gamma_{10}}t_{c_2}^{-1}t_{f_{26}}t_{d_{41}}^{-1}$. Construct an explicit diffeomorphism 
\[
\phi:=t_{m_{6,26}}t_{c_2}t_{f_{26}}t_{m_{6,26}}t_{m_{13,41}}t_{d_{41}}t_{\gamma_{10}}t_{m_{13,41}}
\]
mapping curve pairs \(
(\gamma_{10}, f_{26}) \mapsto (d_{41}, c_2),
\) verifying that  \(t_{\gamma_{10}}t_{c_2}^{-1}t_{f_{26}}t_{d_{41}}^{-1} = [t_{\gamma_{10}}t_{f_{26}}, \phi]\) by conjugation.
 \item $g=4k+1\geq29$: The twist subgroup $\mathcal{T}_g$ is generated by $T$ and $t_{\gamma_{10}}t_{c_2}^{-1}t_{f_{18}}t_{c_{12}}^{-1}$ by ~\cite[Theorem 3.1]{altunoz-pamuk-yildiz2}. Similarly, one can find a diffeomorphism $\varphi$ that maps $(\gamma_{10},f_{18}) \mapsto (c_{12}, c_2)$. We obtain $t_{\gamma_{10}}t_{c_2}^{-1}t_{f_{18}}t_{c_{12}}^{-1}=[t_{\gamma_{10}}t_{f_{18}},\varphi]$.
 \item $g=4k+2\geq50$: For this case, by ~\cite[Theorem 4.3]{altunoz-pamuk-yildiz2}, $\mathcal{T}_g$ is generated by $T$ and $t_{\gamma_{11}}t_{b_3}^{-1}t_{f_{21}}t_{c_{14}}^{-1}$. To show that this element can be written as a single commutator, one can build a diffeomorphism \(\psi\) sending \((\gamma_{11}, f_{21})\) to \((c_{14}, b_3)\), which shows that the element is equal to \([t_{\gamma_{11}}t_{f_{21}}, \psi]\).
 \item $g=4k+3\geq43$: We use the generating set given in ~\cite[Theorem 4.3]{altunoz-pamuk-yildiz2}: the three elements $T$, $t_{\gamma_{12}}t_{b_3}^{-1}t_{f_{21}}t_{u_{37}}^{-1}$ and $t_{b_{2k+1}}t_{a_{1}}^{-1}$. We can explicitly construct diffeomorphisms $\phi_1$ and  $\phi_2$ such that 
\begin{itemize}
        \item \(t_{\gamma_{12}}t_{b_3}^{-1}t_{f_{21}}t_{u_{37}}^{-1} = [t_{\gamma_{12}}t_{f_{21}}, \phi_1]\) by \(\phi_1: (\gamma_{12}, f_{21}) \mapsto (u_{37}, b_3)\).
        \item \(t_{b_{2k+1}}t_{a_{1}}^{-1} = [t_{b_{2k+1}}, \phi_2]\) where \(\phi_2\) maps \(b_{2k+1} \mapsto a_1\).
\end{itemize}
(Here all curves are shown in Figures~\ref{Todd_twist} and ~\ref{Teven_twist}.) See \cite{altunoz-pamuk-yildiz2} for the explicit expressions of the diffeomorphisms $\varphi$, \(\psi\), $\phi_1$ and  $\phi_2$.
\end{itemize}

\begin{theorem}\label{theorem_twist_com}
The twist subgroup $\mathcal{T}_g$ is generated by
\begin{itemize}
    \item two commutators if $g=4k\geq50$, $g=4k+2\geq50$ or $g=4k+1\geq29$, and
    \item three commutators if $g=4k+3\geq43$.
\end{itemize}
\end{theorem}

For smaller values of the genus $g$, the number of required commutator generators increases. The following results provide a bound on the number of commutator generators when $g\geq 8$. The argument of the proof of the following theorem follows the same approach as in Theorem~\ref{theorem_twist_com}.
\begin{theorem}\label{theorem_twist_com_small}
The twist subgroup $\mathcal{T}_g$ is generated by 
  \begin{itemize}
      \item three commutators if $g=2r+2\geq8$ or $g=4k+1\geq9$, and
      \item four commutators if  $g=4k+3\geq11$.
  \end{itemize}
\end{theorem}

\bigskip

\noindent

\section{Open Problems and Future Directions}

Despite significant progress in understanding the generating sets for mapping class groups of nonorientable surfaces, several fundamental questions remain open. This section highlights some key problems and potential directions for future research.
\begin{itemize}
\item \textbf{Generation in Small and Intermediate Genera:} Many results for minimal generation by two or three elements hold for a sufficiently large genus $g$ (e.g., $g \ge 19$ for $\mathrm{Mod}(N_g)$). The behavior for smaller and intermediate genera remains an active area of investigation. A specific unresolved case is for the genus-$4$ surface, where it remains an open question whether $\mathrm{Mod}(N_4)$ can be generated by three torsion elements. Closing the gap between the small genera with known properties and the large genera covered by some results is a primary goal.

\item \textbf{Presentations from Minimal Generators:} A natural and significant next step is to find finite presentations for $\mathrm{Mod}(N_g)$ and $\mathcal{T}_g$ based on the minimal generating sets. For instance, finding a complete set of relations for the two-element generating set of $\mathrm{Mod}(N_g)$ for $g \ge 19$ would be a major breakthrough, deepening our algebraic understanding of these groups.

\item \textbf{Minimal Torsion Generation:}
 While it is known that $\mathrm{Mod}(N_g)$ can be generated by two elements for $g \geq 19$, it remains open whether these generators can both be torsion elements. An analogous question exists for the twist subgroup $\mathcal{T}_g$. It is generated by two elements for large $g$, but it is unknown whether two torsion elements suffice, particularly in the non-perfect cases ($g < 7$).

\item \textbf{Minimal Generation for Punctured Surfaces:} The results for the mapping class groups of punctured nonorientable surfaces, $\mathrm{Mod}(N_{g,p})$, hold primarily for large genera (e.g., $g \ge 14$). Determining the minimal number of generators (and also involution or torsion generators) for all values of $g$ and $p$ is a broad and challenging area for future work.

\item \textbf{Minimal Dehn Twist Generators for the Twist Subgroup:} While Stukow~\cite{stukow1} and Omori~\cite{omori} provided explicit finite generating sets for the twist subgroup $\mathcal{T}_g$ using Dehn twists, the question of minimality is unresolved. Omori's set consists of $g+1$ Dehn twists for $g \ge 4$, but it is not known if this is the smallest possible number of Dehn twist generators.

\item \textbf{Commutator Generators for the Twist Subgroup:} The twist subgroup $\mathcal{T}_g$ is known to be perfect for $g \ge 7$. For very large genera ($g \ge 29$ or $g \ge 43$, depending on the congruence class modulo $4$), it can be generated by two or three commutators. However, the minimal number of commutator generators required for smaller genera where the group is perfect (i.e., for $7 \le g < 29$) has not yet been determined.
\end{itemize}

\printindex 

\end{document}